\documentclass{amsart}

\usepackage{color,graphicx,amssymb,latexsym,amsfonts,txfonts,amsmath,amsthm}
\usepackage{pdfsync,enumitem}
\usepackage{mathrsfs, calrsfs}

\usepackage{tikz}
\usetikzlibrary{matrix}

\usepackage{hyperref}
\hypersetup{
    colorlinks=true,       
    linkcolor=blue,          
    citecolor=blue,        
    filecolor=blue,      
    urlcolor=blue           
}

\newtheorem{theo}{Theorem}
\newtheorem*{theo*}{Theorem}

\newtheorem{coro}{Corollary}

\newtheorem{lemm}{Lemma}

\theoremstyle{remark}
\newtheorem{rema}{\bf Remark}
\newtheorem{example}{\bf Example}

\numberwithin{equation}{section}

\begin{document}

\title{Homology Covers and Automorphisms: Examples}

\author{Rub\'en A. Hidalgo}
\address{Departamento de Matem\'atica y Estad\'{\i}stica, Universidad de La Frontera. Temuco, Chile}
\email{ruben.hidalgo@ufrontera.cl}

\thanks{Partially supported by Projects Fondecyt 1230001 and 1220261}



\begin{abstract}
Let $S$ be a Riemann surface, with a non-abelian fundamental group, and $\widetilde{S}_{k}$ be its $k$-homology cover, where either $k \geq 2$ is an integer of $k=\infty$. The surface 
$\widetilde{S}_{k}$ admits a group of conformal automorphisms $M_{k} \cong {\rm H}_{1}(S;{\mathbb Z}_{k})$ such that $S=\widetilde{S}_{k}/M_{k}$. If $L$ is a group of conformal automorphisms of $S$, then there is a short exact sequence  $1 \to M_{k} \to \widetilde{L}_{k} \to L \to 1$, where $\widetilde{L}_{k}$ is a group of conformal automorphisms of $\widetilde{S}_{k}$. In general, the short exact sequence does not need to be split. This paper explores situations in which the splitting is or is not achieved and we provides examples of both scenarios.
\end{abstract}

\maketitle

\section{Introduction}
In this paper, $S$ will be a connected Riemann surface with a non-abelian fundamental group (we will refer to it as non-exceptional for short); equivalently, $S$ is not homeomorphic to the sphere, the plane, the one-punctured plane, or the torus. The surface $S$ might be of finite type (i.e., $\pi_{1}(S)$ is finitely generated) or of infinite type. In \cite{Ker,Ian}, there is provided the topological classification of surfaces (of finite and infinite types). The non-exceptional condition asserts that the holomorphic universal cover of $S$  is the hyperbolic plane ${\mathbb H}^{2}$. So, up to biholomorphisms, $S={\mathbb H}^{2}/\Gamma$, for some torsion-free Fuchsian group $\Gamma \cong \pi_{1}(S)$. The Fuchsian group $\Gamma$ is uniquely determined by $S$ up to conjugation.

If either (i) $k \geq 2$ is an integer or (ii) $k=\infty$, then we set by $\Gamma_{k}$ the (characteristic) subgroup of $\Gamma$ which is generated by all of its commutators together with, if $k \neq \infty$, all the $k$-powers of all elements of $\Gamma$. The Riemann surface $\widetilde{S}_{k}={\mathbb H}^{2}/\Gamma_{k}$ is the $k$-homology cover of $S$ (if $k=\infty$, then $\widetilde{S}_{\infty}$ is the homology cover). If $S$ is not homeomorphic to a once-punctured closed surface, then $\widetilde{S}_{\infty}$ is known to be homeomorphic to the Loch Ness monster (the surface of infinite genus and one end) \cite{BH} (this fact is reproved in \cite{Biringer}, moreover, if $S$ is of infinite type, then $\widetilde{S}_{k}$ is also homeomorphic to the Loch Ness monster).
When $S$ is of finite type, in \cite{Hidalgo:Homology, Hidalgo:nodedfunction, Maskit:homology}, it was observed that $\widetilde{S}_{\infty}$ determines it uniquely up to biholomorphism (this can be seen as a kind of Fuchsian groups version of Torelli's theorem). For $S$ of infinite type, such a homology-rigidity result is not known to be true. In the case that 
$S$ is a closed Riemann surface, and $k \geq 2$ is finite, then $\widetilde{S}_{k}={\mathbb H}^{2}/\Gamma_{k}$ determines $S$ for suitable values of $k$ \cite{Hidalgo:Fermat} (again, this rigidity property is not known for the case of non-comapct surfaces).

Let us denote by ${\rm Aut}(S)$ and ${\rm Aut}(\widetilde{S}_{k})$ the groups of conformal automorphisms of $S$ and $\widetilde{S}_{k}$, respectively. The group ${\rm Aut}(\widetilde{S}_{k})$ contains a subgroup 
$M_{k}\cong {\rm H}_{1}(S;{\mathbb Z}_{k})$, where ${\mathbb Z}_{k}$ denotes the cyclic group of order $k$, such that $S=\widetilde{S}_{k}/M_{k}$. 
Let ${\rm Aut}_{M_{k}}(\widetilde{S}_{k})$ be the normalizer of $M_{k}$ in ${\rm Aut}(\widetilde{S}_{k})$.

If $S$ is of finite type, then ${\rm Aut}_{M_{\infty}}(\widetilde{S}_{\infty})={\rm Aut}(\widetilde{S}_{\infty})$ \cite{Hidalgo:Homology}. If $S$ is a closed Riemann surface of genus $g \geq 2$, and 
either (i) $k=2$ and $S$ hyperelliptic, or (ii) $k=p^{r}$, where $p > 84(g-1)$ is a prime integer, then ${\rm Aut}_{M_{k}}(\widetilde{S}_{k})={\rm Aut}(\widetilde{S}_{k})$ \cite{Hidalgo:Fermat}.

Let $\pi_{k}:\widetilde{S}_{k} \to S$ be a Galois covering, with deck group $M_{k}$. Then, as $\Gamma_{k}$ is a characteristic subgroup of $\Gamma$, there exists a short exact sequence 
\begin{equation}\label{exacta0}
\begin{array}{l} 
1 \to M_{k}\cong {\rm H}_{1}(S;{\mathbb Z}_{k}) \to {\rm Aut}_{M_{k}}(\widetilde{S}_{k}) \stackrel{\theta_{k}}{\to} {\rm Aut}(S) \to 1,
\end{array}
\end{equation}
such that, for every $\psi \in {\rm Aut}_{M_{k}}(\widetilde{S}_{k})$, it holds that 
$\pi_{k} \circ \psi = \theta_{k}(\psi) \circ \pi_{k}.$

In particular, if $L \leq {\rm Aut}(S)$ and $\widetilde{L}_{k}=\theta_{k}^{-1}(L)$, then the above short exact sequence restricts to a short exact sequence
\begin{equation}\label{exacta1}
\begin{array}{l}
1 \to M_{k}\cong {\rm H}_{1}(S;{\mathbb Z}_{k}) \to \widetilde{L}_{k} \stackrel{\theta_{k}}{\to} L \to 1.
\end{array}
\end{equation}

If $S$ is of finite type (so $L$ is finite), $k \neq \infty$, and the order of $L$ is relatively prime to $k$, then (as a consequence of the Schur-Zassenhaus theorem) the short exact sequence \eqref{exacta1} splits. Since: (i) for $S$ of infinite type the group $M_{k}$ is not finite, and (ii) for $S$ of finite type the group $M_{\infty}$ is also infinite, we cannot use the Schur-Zassenhaus theorem to check if \eqref{exacta1} splits or not in the general case.  In the general situation, the splitting situation does not hold. 

Let us start with the following simple observation, concerning the splitting of the above short exact sequence.

\begin{theo}\label{teo6}
Let $S$ be a non-exceptional Riemann surface, and $L \leq {\rm Aut}(S)$.
For $k \in \{2,3,\ldots\} \cup \{\infty\}$, let $\widetilde{L}_{k}=\theta_{k}^{-1}(L)<{\rm Aut}(\widetilde{S}_{k})$ and $\theta_{k}$ be as in the short exact sequence \eqref{exacta0}.

\begin{enumerate}[leftmargin=*,align=left]
\item If $L=\langle \phi \rangle \cong {\mathbb Z}_{l}$, where $l \geq 2$, then the short exact sequence \eqref{exacta1} splits, so $\widetilde{L}_{k} \cong {\rm H}_{1}(S;{\mathbb Z}_{k}) \rtimes {\mathbb Z}_{l}$, if either (i) or (ii) below holds. 
\begin{enumerate}
\item[(i)] $\phi$ has fixed points.
\item[(ii)] $k$ is finite and relatively prime to $l$.
\end{enumerate}

\item \begin{enumerate}[leftmargin=*,align=left]
\item[(a)] Fix a presentation ${\mathcal P}=\{ \{\phi_{j}\}_{j \in J}: \{R_{i}(\{\phi_{j}\}_{j \in J})=1\}_{i \in I}\}$ for $L$, and fix $\{\psi_{j}\}_{j \in J} \subset \widetilde{L}_{k}$ such that $\theta_{k}(\psi_{j})=\phi_{j}$.

\item[(b)] Let $m_{i} \in M_{k}$ be such that $R_{i}(\{\psi_{j}\}_{j \in J})=m_{i}$, and let 
$M_{k,{\mathcal P},\{\psi_{j}\}_{j \in J}}$ be the minimal $\widetilde{L}_{k}$-invariant subgroup of $M_{k}$ containing all the elements $m_{i}$, where $i \in I$.

\item[(c)] Let us consider the Riemann surface $W_{k,{\mathcal P},\{\psi_{j}\}_{j \in J}}=\widetilde{S}_{k}/ M_{k,{\mathcal P},\{\psi_{j}\}_{j \in J}}$, together its following groups of automorphisms
$$\begin{array}{l}
{\mathcal G}_{k,{\mathcal P},\{\psi_{j}\}_{j \in J}}=\widetilde{L}_{k}/M_{k,{\mathcal P},\{\psi_{j}\}_{j \in J}},\;
 {\mathcal K}_{k,{\mathcal P},\{\psi_{j}\}_{j \in J}}=M_{k}/ M_{k,{\mathcal P},\{\psi_{j}\}_{j \in J}}, \;
\hat{L}=\langle  \psi_{j}: j \in J\rangle/M_{k,{\mathcal P},\{\psi_{j}\}_{j \in J}}.
\end{array}$$

\end{enumerate}

\end{enumerate}

Then the following short exact sequence splits
\begin{equation}\label{exacta2}
\begin{array}{l}
1 \to {\mathcal K}_{k,{\mathcal P},\{\psi_{j}\}_{j \in J}}=M_{k}/M_{k,{\mathcal P},\{\psi_{j}\}_{j \in J}} \to {\mathcal G}_{k,{\mathcal P},\{\psi_{j}\}_{j \in J}}=\widetilde{L}_{k}/M_{k,{\mathcal P},\{\psi_{j}\}_{j \in J}} \stackrel{\theta_{k}}{\to} L \to 1,
\end{array}
\end{equation}
and, in particular,
$$\begin{array}{l}
{\mathcal G}_{k,{\mathcal P},\{\psi_{j}\}_{j \in J}} = {\mathcal K}_{k,{\mathcal P},\{\psi_{j}\}_{j \in J}} \rtimes \hat{L} \cong {\mathcal K}_{k,{\mathcal P},\{\psi_{j}\}_{j \in J}} \rtimes L.
\end{array}$$
\end{theo}

\begin{rema}
Note that case (ii) of part (1), in the above result, is just a consequence of the Schur-Zassenhaus theorem. In part (2), the group $M_{k,{\mathcal P},\{\psi_{j}\}_{j \in J}}$ depends on both ${\mathcal P}$ and the choices of $\psi_{j}$ (see Example \ref{ej1}).\qed
\end{rema}

\begin{coro}\label{coro2}
The short exact sequence \eqref{exacta1} splits, so $\widetilde{L}_{k} \cong M_{k} \rtimes L$, if and only if there is some ${\mathcal P}$ and corresponding choices of $\psi_{j}$ such that $M_{k,{\mathcal P},\{\psi_{j}\}_{j \in J}}=\{1\}$.
\end{coro}

In Section \ref{Sec:ejemplito}, we consider an application to Galois closures (see Corollary \ref{coro3}).

In the case that (i) $L$ is not a cyclic group or (ii) $L$ is a finite cyclic group, either of order not relatively prime to the exponent of $\widehat{\mathcal A}$ or acting freely, it might happen that the short exact sequence \eqref{exacta1} doesn't split.
In Section \ref{Sec:noramificado}, we observe (see Lemma \ref{lemita1}) that for $L \cong {\mathbb Z}_{l}$ a finite cyclic group acting freely on a closed Riemann surface $S$, the splitting situation holds if and only if $k$ is relatively prime to $l$. 
In Section \ref{Sec:S3}, we consider the case when $L \cong {\mathfrak S}_{3}$, the symmetric group of order $6$, is a group of conformal automorphisms of a closed Riemann surface $S$, where the elements of order three act freely, the elements of order two have fixed points, and $S/L$ has genus zero. The Schur-Zassenhaus theorem asserts that, for $k$ relatively prime to $3$, the short exact sequence \eqref{exacta1} splits. In Lemma \ref{lema2}, we observe that, for $k$ divisible by $3$, the short exact sequence never splits. 
Examples \ref{Sec:(m,k)=(2,4)}, \ref{Sec:711}, \ref{Sec:k=2} and \ref{Sec:(n,k)=(5,3)}, describe some non-splitting and splitting cases.

\medskip\noindent
{\bf Notations:} If $\Gamma$ is a group, then we denote by $\Gamma'$ its derived subgroup (i.e., the subgroup generated by all the commutators of $\Gamma$), by $\Gamma^{k}$ the subgroup generated by all the $k$-powers of all the elements of $\Gamma$, and by $\Gamma_{k}$ the subgroup generated by $\Gamma'$ and $\Gamma^{k}$. Note that all of the above subgroups are characteristic ones. We denote by ${\mathbb Z}_{l}$ the cyclic group of order $l$.

\section{Homology coverings of surfaces}\label{Sec:homocovers}
In this section, we discuss some of the general facts concerning the $k$-homology covers of surfaces. 

\subsection{The $k$-homology cover of a Riemann surface}
Let $S$ be a non-exceptional Riemann surface (it might or may not be of finite type). By the uniformization theorem, there is a (torsion-free) Fuchsian group $\Gamma$ such that (up to biholomorphisms) $S={\mathbb H}^{2}/\Gamma$. 
The condition for $S$ to be of finite type is equivalent for $\Gamma$ to be finitely generated, and for it to be analytically finite is that $\Gamma$ is of finite hyperbolic area.
If $k \geq 2$ or $k=\infty$, then the Riemann surface $\widetilde{S}_{k}={\mathbb H}^{2}/\Gamma_{k}$ is called the $k$-homology cover of $S$ (if $k=\infty$, then it is also called the homology cover).
If $S$ is not homeomorphic to a once-punctured closed surface, then $\widetilde{S}_{\infty}$ is homeomorphic to the Loch Ness monster, i.e., the unique (up to homeomorphisms) surface of infinite genus and one end (see Corollary 6.2 in \cite{Ata} for the case of closed Riemann surfaces, and \cite{BH} for the general situation).

\subsection{A connection to  the jacobian variety}\label{torelli}
Let us assume $S$ is a closed Riemann surface of genus $g \geq 1$. Its ${\mathbb C}$-vector space ${\rm H}^{1,0}(S)$ of holomorphic one-forms has dimension $g$. Let us denote by ${\rm H}^{1,0}(S)^{*}$ the dual space of ${\rm H}^{1,0}(S)$. There is a natural lattice embedding of ${\rm H}_{1}(S;{\mathbb Z})$ into ${\rm H}^{1,0}(S)$ which is induced by integration of one-forms on loops. 
The quotient $JS={\rm H}^{1,0}(S)^{*}/{\rm H}_{1}(S;{\mathbb Z})$, called the jacobian variety of $S$,  is a complex torus of dimension $g$, with a principal polarization induced by the intersection form in homology. Each point $p \in S$ induces a holomorphic embedding $\iota_{p}:S \hookrightarrow JS$, which is defined by $\iota(q)=\int_{\gamma} \cdot$, where $\gamma$ is an arc connecting $p$ to $q$.
If $\pi:{\rm H}^{1,0}(S)^{*} \to J(S)$ is a Galois covering, with deck group ${\rm H}_{1}(S;{\mathbb Z})$, then $\pi^{-1}(\iota_{p}(S))$ is biholomorphic to $\widetilde{S}_{\infty}$. Moreover, the restriction of the action of ${\rm H}_{1}(S;{\mathbb Z})$ on this lifting coincides with $M_{\infty}$. Torelli's theorem asserts that the Riemann surface $S$ is uniquely determined, up to biholomorphisms, by $JS$ (as principally polarized variety).

\subsection{Some rigidity theorems}
As seen in Section \ref{torelli}, two closed Riemann surfaces of genera at least one are biholomorphically equivalent if and only if their jacobian varieties are isomorphic as abelian varieties. 
A Fuchsian group version of Torelli's theorem is due to B. Maskit \cite{Maskit:homology}.

\begin{theo}[Torelli's co-compact Fuchsian groups type theorem \cite{Maskit:homology}]\label{teo0}
Two closed Riemann surfaces of genus at least two are biholomorphically equivalent if and only if their homology covers are biholomorphically equivalent.
\end{theo}

There seems to be no version of Torelli's theorem, at least known to the author, for non-closed Riemann surfaces. On the other hand, the above Fuchsian version has been generalized to certain classes of non-elementary Kleinian groups. 

\begin{theo}[Torelli's Kleinian groups type theorem \cite{Hidalgo:Homology, Hidalgo:nodedfunction}]\label{teo1}
Let $K$ and $\widehat{K}$ be two torsion-free, finitely generated, and non-elementary Kleinian groups, each admitting an invariant connected component of its region of discontinuity. 
If $K'=\widehat{K}'$, then $K=\widehat{K}$. 
\end{theo}

\begin{rema}
If $K$ is a torsion-free, finitely generated, non-elementary Kleinian group admitting a $K$-invariant connected component
$\Delta$ of its region of discontinuity, then Ahlfors' finiteness theorem \cite{A1,A2} asserts that the quotient Riemann surface $\Delta/K$ is analytically finite.\qed
\end{rema}

If, in the above theorem, we assume the groups to be Fuchsian groups, then we obtain the following Fuchsian version of Torelli's theorem for hyperbolic finite type Riemann surfaces.

\begin{coro}
Two (hyperbolic) finite-type Riemann surfaces are biholomorphically equivalent if and only if their homology covers are biholomorphically equivalent.
\end{coro}

\begin{rema}
It is not known if the above Kleinian rigidity theorems are valid in the case of torsion-free infinitely generated Fuchsian groups. It is not even clear whether the equality of the derived subgroups asserts that the groups should be commensurable or not.\qed
\end{rema}

Since $\bigcap_{k \geq 2} \Gamma_{k}=\Gamma'$, it seems natural to ask if the surface $S$ is completely determined (up to biholomorphisms) by its $k$-homology cover, for 
$k \geq 2$.  If $S$ is a closed Riemann surface, then a partial answer is provided by the following.

\begin{theo}[\cite{Hidalgo:Fermat}]\label{teo2}
Two closed Riemann surfaces of genus $g \geq 2$ are biholomorphic if and only if their $k$-homology covers are biholomorphic in any of the following situations.
\begin{enumerate}[leftmargin=*,align=left]
\item[(i)] $k=\infty$.
\item[(ii)] $k=p^{r}$, where $p>84(g-1)$ is a prime integer and $r \geq 1$.
\item[(iii)] $k=2$ and one of the surfaces is hyperelliptic.
\end{enumerate}
\end{theo}

\subsection{The rigidity in the presence of torsion}
The above rigidity results are not generally valid for closed Riemann orbifolds \cite{Hidalgo:MoscowHomology}. For certain classes of 
Riemann orbifolds, the same commutator rigidity result holds.

\begin{theo}[\cite{Hidalgo:MoscowHomology, HKLP}]\label{teo3}
Let $n_{1}, n_{2}, k_{1}, k_{2} \geq 2$ be integers such that $(n_{j}-1)(k_{j}-1)>2$.
Let $\Gamma$ and $\widehat{\Gamma}$ be two co-compact Fuchsian groups of respective signatures $(0;k_{1},\stackrel{n_{1}+1}{\cdots}, k_{1})$ and
$(0;k_{2},\stackrel{n_{2}+1}{\cdots}, k_{2})$. If $\Gamma'=\widehat{\Gamma}'$, 
then $\Gamma=\widehat{\Gamma}$. 
\end{theo}

\section{Splitting results: proof of Theorem \ref{teo6}}\label{Sec:mainresults}
In this section, $S$ will denote a non-exceptional Riemann surface (either of finite type or infinite type), and $L \leq {\rm Aut}(S)$ (note that $L$ could be infinite, but it must be countable). 

By abuse of notation, we will identify ${\rm H}_{1}(S;{\mathbb Z})$ with $M_{\infty}$, and if $x \in M_{\infty}={\rm H}_{1}(S;{\mathbb Z})$, then we will use the same letter to denote the element $\widetilde{\theta}_{k}(x) \in M_{k}$, where $\widetilde{\theta}_{k}:M_{\infty} \to M_{k}=M_{\infty}/M_{\infty}^{k}$ is the natural projection homomorphism. 

We will denote by $\widehat{\pi}_{k}:\widetilde{S}_{\infty} \to \widetilde{S}_{k}$ a Galois covering with deck group $M_{\infty}^{k}$ such that $\pi_{\infty}=\pi_{k} \circ \widehat{\pi}_{k}$.
The surjective homomorphism $\widetilde{\theta}_{k}:M_{\infty} \to M_{k}$ extends to a surjective homomorphism $\widetilde{\theta}_{k}:\widetilde{L}_{\infty} \to \widetilde{L}_{k}=\widetilde{L}_{\infty}/M_{\infty}^{k}$ such that $\theta_{\infty}=\theta_{k} \circ \widetilde{\theta}_{k}$, and
$\widehat{\pi}_{k} \circ \psi=\widetilde{\theta}_{k}(\psi) \circ \widehat{\pi}_{k}$. If $\psi \in \widetilde{L}_{\infty}$, then many times we will use the same letter $\psi$ or the letter $\Psi$ to denote  $\widetilde{\theta}(\psi) \in \widetilde{L}_{k}$. 

\begin{rema}
If $S$ is a closed Riemann surface of genus $g \geq 2$, then ${\rm H}_{1}(S;{\mathbb Z}) \cong {\mathbb Z}^{2g}$. If $S$ is non-closed of finite type, say of genus $g \geq 0$ and $n \geq 1$ boundary components ($2g-2+n>0$), then ${\rm H}_{1}(S;{\mathbb Z}) \cong {\mathbb Z}^{2g+n-1}$. If $S$ is of infinite type, then ${\rm H}_{1}(S;{\mathbb Z})$ is a free abelian group of infinite rank.\qed
\end{rema}

We fix a (minimal) set $\{\gamma_{s}\}_{s \in E}$ of generators of $M_{\infty}= {\rm H}_{1}(S;{\mathbb Z})$ (i.e., a set of generators such that the only relations are consequences of commutators). For instance, if $S$ is of genus $g$ and has $n$ punctures, then such a minimal set of generators consists of $2g+n-1$ elements, say $a_{1},\ldots, a_{g}, b_{1}, \ldots, b_{g}, c_{1}, \ldots, c_{n-1}$ (if $n=0$, then we delete the generators $c_{j}$).

\subsection{Proof of Part (1): The cyclic case}
Let $L=\langle \phi \rangle \cong {\mathbb Z}_{l}$, $k \in \{2,\ldots\}\cup\{\infty\}$, and $\psi \in \widetilde{L}_{k}$ such that $\theta_{k}(\psi)=\phi$. The others elements of $\widetilde{L}_{k}$ which are send to $\phi$ under $\theta_{k}$ are all of the form $\psi_{\alpha}=\alpha \circ \psi$, where $\alpha \in M_{k}$. The question is if one may find $\alpha \in M_{k}$ such that $\psi _{\alpha}$ has the same order as $\phi$. 

If $k \geq 2$ is finite and $l$ and $k$ are relatively primes, then the result follows from the Schur-Zassenhaus theorem. So, let us now assume that $\phi$ has fixed points (and we admit the possibility that $k=\infty$). Let $p \in \widetilde{S}_{k}$ be such that $\pi_{k}(p)$ is fixed by $\phi$. Let $u \in \widetilde{L}_{k}$ be such that $\theta_{k}(u)=\phi$. Then, $\pi_{k}(u(p))=\phi(\pi_{k}(p))=\pi_{k}(p)$. So, there is some $\alpha \in M_{k}$ such that $\alpha(u(p))=p$. Set $\psi=\alpha \circ u \in \widetilde{L}_{k}$. In this case, $\theta_{k}(\psi)=\theta_{k}(\alpha \circ u)=\theta_{k}(u)=\phi$, and $\psi(p)=p$. 
As $\phi^{l}=1$, then $\psi^{l} \in M_{k}$. Since $\Psi$ has fixed points and $M_{k}$ acts freely on $\widetilde{S}_{k}$, it follows that $\psi^{l}=1$.
\qed

\begin{rema}
In the case that $\phi$ either has no fixed points (or if its order is not relatively prime to $k$ for the case $k$ finite), it might be that there is no $\psi \in \widetilde{L}_{k}$ of the same order as $\phi$ with $\theta_{k}(\psi)=\phi$ (see Section \ref{Sec:noramificado}). In the following example, we consider a case when $\phi$ has infinite order.\qed
\end{rema}

\begin{example}
Let $S$ be a Riemann surface that is homeomorphic to the Jacob ladder (the unique, up to homeomorphisms, surface of infinite genus with exactly two ends, each one accumulated by genus) admitting a conformal automorphism $\phi$ of infinite order (which is repelling from one end and attracting to the other end). Note that $\phi$ has no fixed points on $S$. In this case, ${\rm H}_{1}(S;{\mathbb Z})$ has a set of generators $\{\delta\} \cup \{a_{j}, b_{j}: j \in {\mathbb Z}\}$,  where 
(i) $a_{j}, b_{j}$ are represented by two oriented simple loops intersecting exactly at one point, whose commutator $[a_{j},b_{j}]$ is represented by a simple loop that cut-off a torus, and 
$\phi(a_{j})=a_{j+1}$, $\phi(b_{j})=b_{j+1}$, and 
(ii) $\delta$ is represented by an oriented simple loop that separates one end from the other and such that the loops $\delta, [a_{0},b_{0}]$ and $\phi(\delta)$ bounds a pant. 
The action of $\phi$ on homology is given by $\phi_{*}(\delta)=\delta$,  $\phi_{*}(a_{j})=a_{j+1}$, $\phi_{*}(b_{j})=b_{j+1}$, and
$\widetilde{L}_{\infty}=M_{\infty} \rtimes \langle \psi \rangle \cong {\rm H}_{1}(S;{\mathbb Z}) \rtimes {\mathbb Z}$, where $\psi \in {\rm Aut}(\widetilde{S}_{\infty})$ is such that $\theta_{\infty}(\psi)=\phi$.\qed
\end{example}

\subsection{Proof of Part (2): The general case}\label{Sec:251}
As above, $k \in \{2,3,\ldots\} \cup \{\infty\}$, and $L$ is a (necessarily countable or finite) group of conformal automorphisms of a non-exceptional Riemann surface $S$.

Let us fix a presentation ${\mathcal P}$ of $L$, say with a set of generators $\{\phi_{j}\}_{j \in J}$ and a set of relations $\{R_{i}(\{\phi_{j}\})=1\}_{i \in I}$ (where $I$ and $J$ are either finite of countable infinite sets).

Fix some choices for 
$\psi_{j} \in \widetilde{L}_{k}$ such that $\theta_{k}(\psi_{j})=\phi_{j}$, for $j \in J$, and let 
$m_{i} \in M_{k}$ be such that $R_{i}(\{\psi_{j}\})=m_{i}$, for $i \in I$.

Let $M_{k,{\mathcal P},\{\psi_{j}\}_{j \in J}}$ be the minimal $\widetilde{L}_{k}$-invariant subgroup of $M_{k}$ containing the elements $m_{i}$, where $i \in I$, determined by the previous fixed presentation ${\mathcal P}$ of $L$.
Note that $M_{k,{\mathcal P},\{\psi_{j}\}_{j \in J}}$ is a normal subgroup of the subgroup $\langle \psi_{j}: j \in J \rangle \leq \widetilde{L}_{k}$, and $\theta_{k}(\langle \psi_{j}: j \in J \rangle)=L$.

The subgroup $M_{k,{\mathcal P},\{\psi_{j}\}_{j \in J}}$ determines a Riemann surface $W_{k,{\mathcal P},\{\psi_{j}\}_{j \in J}}=\widetilde{S}/M_{k,{\mathcal P},\{\psi_{j}\}_{j \in J}}$, 
an abelian Galois covering $\pi_{k,{\mathcal P},\{\psi_{j}\}_{j \in J}}:\widetilde{S}_{k} \to W_{k,{\mathcal P},\{\psi_{j}\}_{j \in J}}$, with deck group  $M_{k,{\mathcal P},\{\psi_{j}\}_{j \in J}}$, and an abelian Galois covering  $\eta_{k,{\mathcal P},\{\psi_{j}\}_{j \in J}}:W_{k,{\mathcal P},\{\psi_{j}\}_{j \in J}} \to S$, with deck group ${\mathcal K}_{k,{\mathcal P},\{\psi_{j}\}_{j \in J}}=M_{k}/ M_{k,{\mathcal P},\{\psi_{j}\}_{j \in J}}$, such that $\pi_{k}=\eta_{k,{\mathcal P},\{\psi_{j}\}_{j \in J}} \circ \pi_{k,{\mathcal P},\{\psi_{j}\}_{j \in J}}$.

Since  $M_{k,{\mathcal P},\{\psi_{j}\}_{j \in J}}$ is $\widetilde{L}_{k}$-invariant, the group $\widetilde{L}_{k}$ induces a group ${\mathcal G}_{k,{\mathcal P},\{\psi_{j}\}_{j \in J}}:=\widetilde{L}_{k}/M_{k,{\mathcal P},\{\psi_{j}\}_{j \in J}}$
of conformal automorphism of $W_{k,{\mathcal P},\{\psi_{j}\}_{j \in J}}$ such that $W_{k,{\mathcal P},\{\psi_{j}\}_{j \in J}}/{\mathcal G}_{k,{\mathcal P},\{\psi_{j}\}_{j \in J}}=S/L$. 

The group $\langle \psi_{j}: j \in J \rangle$ induces, under $\pi_{k}$, the group $L$, and, under $\pi_{k,{\mathcal P},\{\psi_{j}\}_{j \in J}}$,  the group $\hat{L}$.
By the choice of the elements $m_{i}$ and the subgroup $M_{k,{\mathcal P},\{\psi_{j}\}_{j \in J}}$, we also have that $\hat{L} \cong L$. This permits us to observe that ${\mathcal G}_{k,{\mathcal P},\{\psi_{j}\}_{j \in J}} = {\mathcal K}_{k,{\mathcal P},\{\psi_{j}\}_{j \in J}} \rtimes \hat{L}.$\qed

\section{Computing a presentation of $\widetilde{L}_{k}$}\label{Sec:presentationgeneral}
Let $S$ be a non-exceptional Riemann surface and $L \leq {\rm Aut}(S)$. Below, we describe how to compute a presentation of $\widetilde{L}_{k}$ starting from a presentation for $L$. Then, we make this explicit for the case when $S$ is a closed Riemann surface as we will need it for constructing our examples.

\subsection{The induced action on homology}
Each element $\phi \in L$ induces an automorphism $\phi_{*}$ of ${\rm H}_{1}(S;{\mathbb Z})$. 
If $\psi \in \widetilde{L}_{\infty}$ is such that $\theta_{\infty}(\psi)=\phi \in L$, then 
the conjugation action of $\psi$ on $M_{\infty}$ reads as follows:
$$\begin{array}{l}
\psi \circ x \circ \psi^{-1}=\phi_{*}(x), \; x \in M_{\infty}.
\end{array}
$$

If $\widetilde{\phi} \in \widetilde{L}_{k}$ is such that $\theta_{k}(\widetilde{\phi})=\phi$, then the above permits us to observe that 
the conjugation action of $\widetilde{\phi}$ on $M_{k}$ reads as follows:
$$\begin{array}{l}
\widetilde{\phi} \circ y \circ \widetilde{\phi}^{-1}=\widetilde{\theta}_{k}(\psi \circ x \circ \psi^{-1})=\widetilde{\theta}_{k}(\phi_{*}(x)), \; x \in M_{\infty}, \; \widetilde{\theta}_{k}(x)=y.
\end{array}
$$

Let us consider a (minimal) set $\{\gamma_{s}\}_{s \in E}$ of generators of $M_{\infty}= {\rm H}_{1}(S;{\mathbb Z})$, a 
presentation ${\mathcal P}=\langle \{\phi\}_{j \in J}: \{R_{i}(\{\phi_{j}\}_{j \in J})\}_{i \in I} \rangle$ of $L$, and choices $\{\psi_{j}\}_{j \in J} \subset \widetilde{L}_{k}$ such that $\theta_{k}(\psi_{j})=\phi_{j}$.

\subsection{Computing a presentation of $\widetilde{L}_{k}$}
The above observations permit us to obtain the following presentations:
$$\begin{array}{l}
\widetilde{L}_{\infty}=\langle \{\gamma_{s}\}_{s \in E}, \{\psi_{j}\}_{j \in J}: [\gamma_{j_{1}},\gamma_{j_{2}}]=1, \; (j_{1}, j_{2} \in E);\\
R_{i}(\{\psi_{j}\})=m_{i}, (i \in I); \; 
\psi_{j} \circ \gamma_{s} \circ \psi_{j}^{-1}=\phi_{*}(\gamma_{s}), (s \in E, \; j \in J)
\rangle.
\end{array}
$$
and, for $k \geq 2$ finite,
$$\begin{array}{l}
\widetilde{L}_{k}=\langle \{\gamma_{s}\}_{s \in E}, \{\psi_{j}\}_{j \in J)}: \gamma_{s}^{k}=1, \; (s \in E);\; 
[\gamma_{j_{1}},\gamma_{j_{2}}]=1, \; (j_{1}, j_{2} \in E);\\
R_{i}(\{\psi_{j}\})=m_{i}, (i \in I);\;
\psi_{j} \circ \gamma_{s} \circ \psi_{j}^{-1}=\phi_{*}(\gamma_{s}), (s \in E, \; j \in J)
\rangle.
\end{array}
$$

\subsection{The closed Riemann surface case}\label{Sec:casocerrado}
Let us make explicit the above in the case that $S$ is a closed Riemann surface of genus $g \geq 2$. 
We fix a set of $2g$ generators of $M_{\infty}$ (which is not assumed to be symplectic)
$$M_{\infty}=\langle a_{1},\ldots,a_{g}, b_{1}, \ldots, b_{g} \rangle = {\rm H}_{1}(S;{\mathbb Z}).$$

In this case, $|{\rm Aut}(S)| \leq 84(g-1)$ \cite{Hurwitz}, so $L \leq {\rm Aut}(S)$ is a finite group.
Let us consider a presentation of $L$ as
$$\begin{array}{l}
{\mathcal P}=\langle \phi_{1},\ldots,\phi_{l}: R_{1}(\phi_{1},\ldots,\phi_{l})=\cdots=R_{n}(\phi_{1},\ldots,\phi_{l})=1\rangle,
\end{array}
$$
let us fix some choices 
$\psi_{1},\ldots, \psi_{l} \in \widetilde{L}_{\infty}$ such that $\theta_{\infty}(\psi_{j})=\phi_{j}$, for $j=1,\ldots,l$, and 
let $m_{1},\ldots,m_{n} \in M_{\infty}$ be such that $R_{j}(\psi_{1},\ldots,\psi_{l})=m_{j}$.

In this case, we obtain the following presentations:
$$\begin{array}{l}
\widetilde{L}_{\infty}=\langle a_{1},\ldots,a_{g}, b_{1}, \ldots, b_{g},  \psi_{1},\ldots,\psi_{l}: [a_{i},a_{j}]=[a_{i},b_{j}]=[b_{i},b_{j}]=1, \; (i,j \in \{1,\ldots,g\});\\
R_{1}(\psi_{1},\ldots,\psi_{l})=m_{1}, \ldots, R_{n}(\psi_{1},\ldots,\psi_{l})=m_{n};\\
\psi_{j} \circ a_{i} \circ \psi_{j}^{-1}=\phi_{*}(a_{i});\;  \psi_{j} \circ b_{i} \circ \psi_{j}^{-1}=\phi_{*}(b_{i}), i \in \{1,\ldots,g\}, \; j \in \{1,\ldots,l\}
\rangle.
\end{array}
$$
and, for $k \geq 2$ finite, 
$$\begin{array}{l}
\widetilde{L}_{k}=\langle a_{1},\ldots,a_{g}, b_{1}, \ldots, b_{g},  \psi_{1},\ldots,\psi_{l}:\\
a_{1}^{k}=\cdots=a_{g}^{k}=b_{1}^{k}=\cdots=b_{g}^{k}=1;
[a_{i},a_{j}]=[a_{i},b_{j}]=[b_{i},b_{j}]=1, \; (i,j \in \{1,\ldots,g\});\\
R_{1}(\psi_{1},\ldots,\psi_{l})=m_{1}, \ldots, R_{n}(\psi_{1},\ldots,\psi_{l})=m_{n};\\
\psi_{j} \circ a_{i} \circ \psi_{j}^{-1}=\phi_{*}(a_{i});\;  \psi_{j} \circ b_{i} \circ \psi_{j}^{-1}=\phi_{*}(b_{i}), i \in \{1,\ldots,g\}, \; j \in \{1,\ldots,l\}
\rangle.
\end{array}
$$

\begin{example}\label{ej0}
Let $S$ be a closed Riemann surface of genus three admitting a conformal automorphism $\phi$ of order four with no fixed points, but $\phi^{2}$ having exactly four fixed points. So, $S/L$ has genus one and exactly two cone points of order two. In this situation, we consider the presentation ${\mathcal P}=\langle \phi: \phi^{4}=1\rangle$.
There is a symplect basis $a_{1}, a_{2}, a_{3}, b_{1}, b_{2}, b_{3}$ for $M_{\infty}$ such that the action of $\phi$ is given by:
$$\begin{array}{l}
\phi_{*}(a_{1})=a_{2}, \; \phi_{*}(a_{2})=a_{1}^{-1}, \; \phi_{*}(a_{3})=a_{3}, \; 
\phi_{*}(b_{1})=b_{2}, \; \phi_{*}(b_{2})=b_{1}^{-1}, \; \phi_{*}(b_{3})=b_{3}.
\end{array}
$$

There is a $\psi \in \widetilde{L}_{k}$ such that $\theta_{k}(\psi)=\phi$ and $\psi^{4}=b_{3}$. So, in this case, 
$$\begin{array}{l}
\widetilde{L}_{k}=\langle a_{1},a_{2},a_{3},b_{1},b_{2},b_{3},\psi: a_{1}^{k}=a_{2}^{k}=a_{3}^{k}=b_{1}^{k}=b_{2}^{k}=b_{3}^{k}=1, \; \psi^{4}=b_{3},\\
\; [a_{i},a_{j}]=[a_{i},b_{j}]=[b_{i},b_{j}]=1 \; (i,j=1,2,3),\\
\psi \circ a_{1} \circ \psi^{-1}=a_{2}, \; \psi \circ a_{2} \circ \psi^{-1}=a_{1}^{-1},\; \psi \circ a_{3} \circ \psi^{-1}=a_{3},\\
\psi \circ b_{1} \circ \psi^{-1}=b_{2}, \; \psi \circ b_{2} \circ \psi^{-1}=b_{1}^{-1},\; \psi \circ b_{3} \circ \psi^{-1}=b_{3}
\rangle,\\
M_{k,{\mathcal P},\{\psi\}}=\langle b_{3} \rangle \cong {\mathbb Z}_{k},\\
{\mathcal K}_{k,{\mathcal P},\{\psi\}}=\langle \hat{a}_{1},\hat{a}_{2},\hat{a}_{5},\hat{b}_{1},\hat{b}_{2} \rangle \cong {\mathbb Z}^{5}_{k},\\
{\mathcal G}_{k,{\mathcal P},\{\psi\}}=\langle \hat{a}_{1},\hat{a}_{2},\hat{a}_{3},\hat{b}_{1},\hat{b}_{2},\psi: \hat{a}_{1}^{k}=\hat{a}_{2}^{k}=\hat{a}_{3}^{k}=\hat{b}_{1}^{k}=\hat{b}_{2}^{k}=\psi^{4}=1,\\
\; [\hat{a}_{i},\hat{a}_{j}]=[\hat{a}_{i},\hat{b}_{j}]=[\hat{b}_{i},\hat{b}_{j}]=1 \; (i=1,2,3; j=1,2),\\
\psi \circ \hat{a}_{1} \circ \psi^{-1}=\hat{a}_{2}, \; \psi \circ \hat{a}_{2} \circ \psi^{-1}=\hat{a}_{1}^{-1},\; \psi \circ \hat{a}_{3} \circ \psi^{-1}=\hat{a}_{3},\\
\psi \circ \hat{b}_{1} \circ \psi^{-1}=\hat{b}_{2}, \; \psi \circ \hat{b}_{2} \circ \psi^{-1}=\hat{b}_{1}^{-1}\rangle = \langle \hat{a}_{1},\hat{a}_{2},\hat{a}_{3},\hat{b}_{1},\hat{b}_{2}\rangle \rtimes \langle \psi \rangle \cong {\mathbb Z}_{k}^{5} \rtimes {\mathbb Z}_{4}.
\end{array}
$$

Any other choice for $\rho \in \widetilde{L}_{k}$ such that $\theta_{k}(\rho)=\phi$ is of the form 
$$\rho=\psi \circ a_{1}^{l_{1}} \circ a_{2}^{l_{2}} \circ a_{3}^{l_{3}} \circ b_{1}^{n_{1}} \circ b_{2}^{n_{2}} \circ b_{3}^{n_{3}}.$$ 
In this case,
$\rho^{4}=a_{3}^{4l_{3}} \circ b_{3}^{4n_{3}+1}$, so 
$M_{k,{\mathcal P},\{\rho\}}=\langle a_{3}^{4l_{3}} \circ b_{3}^{4n_{3}+1} \rangle$.

If we take $k=6$, $l_{3}=0$ and $n_{3}=2$, then $\rho=\psi \circ b_{3}^{2}$, $\rho^{4}=b_{3}^{3}$,
$M_{6,{\mathcal P},\{\rho\}}=\langle b_{3}^{3} \rangle \cong {\mathbb Z}_{2}$, 
${\mathcal K}_{k,{\mathcal P},\{\psi\}}=\langle \tilde{a}_{1},\tilde{a}_{2},\tilde{a}_{5},\tilde{b}_{1},\tilde{b}_{2},\tilde{b}_{3} \rangle \cong {\mathbb Z}_{6}^{5} \times {\mathbb Z}_{2}$, and 
$$\begin{array}{l}
{\mathcal G}_{6,{\mathcal P},\{\rho\}}=\langle \tilde{a}_{1},\tilde{a}_{2},\tilde{a}_{5},\tilde{b}_{1},\tilde{b}_{2},\tilde{b}_{3},\rho: 
\tilde{a}_{1}^{6}=\tilde{a}_{2}^{6}=\tilde{a}_{3}^{6}=\tilde{b}_{1}^{6}=\tilde{b}_{2}^{6}=\tilde{b}_{3}^{3}=\rho^{4}=1,\\
\; [\tilde{a}_{i},\tilde{a}_{j}]=[\tilde{a}_{i},\tilde{b}_{j}]=[\tilde{b}_{i},\tilde{b}_{j}]=1 \; (i,j=1,2,3),\;
\rho \circ \tilde{a}_{1} \circ \rho^{-1}=\tilde{a}_{2}, \; \rho \circ \tilde{a}_{2} \circ \rho^{-1}=\tilde{a}_{1}^{-1},\\ 
\rho \circ \tilde{a}_{3} \circ \rho^{-1}=\tilde{a}_{3},\;
\rho \circ \tilde{b}_{1} \circ \rho^{-1}=\tilde{b}_{2}, \; \rho \circ \tilde{b}_{2} \circ \rho^{-1}=\tilde{b}_{1}^{-1}, \; \rho \circ \tilde{b}_{3} \circ \rho^{-1}=\tilde{b}_{3} \rangle=\\
=\langle \tilde{a}_{1},\tilde{a}_{2},\tilde{a}_{5},\tilde{b}_{1},\tilde{b}_{2},\tilde{b}_{3} \rangle \rtimes \langle \rho \rangle \cong ({\mathbb Z}_{6}^{5} \times {\mathbb Z}_{2}) \rtimes {\mathbb Z}_{4}.
\end{array}
$$
\qed
\end{example}

\subsection{Gilman's adapted basis}
As seen above, to obtain the explicit form of the presentation of $\widetilde{L}_{\infty}$, we need to indicate the values of $\phi_{*}(x)$, for every $x \in {\rm H}_{1}(S;{\mathbb Z})$. 
Let us consider the case when $L=\langle \phi \rangle \cong {\mathbb Z}_{p}$, where $p \geq 2$ is a prime integer. In \cite{Gilman}, J. Gilman provided the existence of a nice set of generators of ${\rm A}_{1}(S;{\mathbb Z})$ for the action of $L$.

A basis ${\mathcal B}$ for ${\rm H}_{1}(S;{\mathbb Z})$ is called adapted to $L$ if for every $x \in {\mathcal B}$ either
\begin{enumerate}[leftmargin=*,align=left]
\item $\phi^{j}_{*}(x) \in {\mathcal B}$, for all $j=0,\ldots,p-1$, or
\item $\phi^{j}_{*}(x)$ is in ${\mathcal B}$ for all $j=0,\ldots,p-2$ and $\phi^{p-1}(x) =-\sum_{r=0}^{p-2} \phi^{r}_{*}(x)$, or
\item $x=\phi^{k}_{*}(y)$, where $0 \leq k \leq p-2$ and $y \in {\mathcal B}$ satisfies (2).
\end{enumerate}

\begin{theo}[Gilman's existence of adapted basis \cite{Gilman}]
If $L=\langle \phi \rangle \cong {\mathbb Z}_{p}$, $p \geq 2$ is a prime integer, then there exists a homology basis adapted to $L$.
\end{theo}

The above permits us to find explicit presentations for $\widetilde{L}_{k}$.
In the next sections, we describe some explicit examples where $S$ is a closed Riemann surface.

\section{An application to Galois closures}\label{Sec:ejemplito}
Let $S$ be a non-exceptional Riemann surface (either of finite type or infinite type), $L \leq {\rm Aut}(S)$ and $\widetilde{\varphi}:S \to S/L$ be a Galois (possible branched) covering with deck group $L$.

Let $U$ be a Riemann surface and $\pi_{S}:U \to S$ be a Galois (unbranched) covering, with deck group an abelian group $\widehat{\mathcal A}$. The Galois closure of $\widetilde{\varphi} \circ \pi_{S}:U \to S/L$ is a minimal Galois (branched) covering $\widetilde{\psi}=:Z \to S/L$, say with deck group ${\mathcal G}$, that factors through $\widetilde{\varphi} \circ \pi_{S}$.

In terms of the $k$-homology cover, the above is described as follows.
Let $k \geq 2$ be any multiple of the exponent of $\widehat{\mathcal A}$ (or $k=\infty$). Then there exists a subgroup $N_{1} \leq M_{k}$ such that $U=\widetilde{S}_{k}/N_{1}$ and $\pi_{S}$ is induced by the inclusion of $N_{1}$ in $M_{k}$ (so, $\widehat{\mathcal A} \cong M_{k}/N_{1}$). 

Let $N_{2}={\rm Core}_{\widetilde{L}_{k}}(N_{1})$ be the core of $N_{1}$ in $\widetilde{L}_{k}$; so $N_{2}$ is the maximal $\widetilde{L}_{k}$-invariant subgroup of $N_{1}$. Then $Z=\widetilde{S}_{k}/N_{2}$, $\widetilde{\psi}$ is induced by the inclusion of $N_{2}$ in $\widetilde{L}_{k}$, and  ${\mathcal G}=\widetilde{L}_{k}/N_{2}$.

The group ${\mathcal G}$ contains the abelian subgroups ${\mathcal U}=N_{1}/N_{2} \lhd {\mathcal K}=M_{k}/N_{2}$. In this case, $U=Z/{\mathcal U}$ and $S=Z/{\mathcal K}$. 
The surjective homomorphism $\theta_{k}:\widetilde{L}_{k} \to L$ induces a surjective homomorphism $\theta:{\mathcal G} \to L$ making the below a short exact sequence
$$1 \to {\mathcal K}  \to {\mathcal G} \stackrel{\theta}{\to} L \to 1.$$

In many cases, ${\mathcal G} \cong {\mathcal K} \rtimes L$ (\cite{B-F, CHR, Diaz-Donagi, G-Sh, Kanev, Recillas, Vetro}). As a consequence of Theorem \ref{teo6}, we have the following.

\begin{coro}\label{coro3}
Let $S$ be a non-exceptional Riemann surface and $\widetilde{\varphi}:S \to T$ be a (branched) Galois covering, with deck group $L\leq {\rm Aut}(S)$.

Let $\pi_{S}:U \to S$ be an unbranched abelian covering, with abelian deck group $\widehat{\mathcal A}$. Let $k \geq 2$ be the exponent of $\widehat{\mathcal A}$. Let $N_{1} \leq M_{k}$ be such that $U=\widetilde{S}_{k}/N_{1}$, $N_{2}={\rm Core}_{\widetilde{L}_{k}}(N_{1})$, $Z=\widetilde{S}_{k}/N_{2}$ and ${\mathcal K}=M_{k}/N_{2}$. 

Then the Galois closure of $\widetilde{\varphi} \circ \pi_{S}$ is given by the Galois covering $\widetilde{\psi}:Z \to L$, induced by the inclusion of $N_{2}$ in $\widetilde{L}_{k}$, and its 
corresponding Galois deck group is ${\mathcal G}=\widetilde{L}_{k}/N_{2}$.

\begin{enumerate}[leftmargin=*,align=left]
\item If $L=\langle \phi \rangle \cong {\mathbb Z}_{l}$ and either (i) $\phi$ has fixed points, or (ii) $l$ is relatively prime to $k$, then  ${\mathcal G} \cong {\mathcal K} \rtimes {\mathbb Z}_{l}$.

\item Assume that there is a presentation ${\mathcal P}=\langle \{\phi\}_{j \in J}: \{R_{i}(\{\phi_{j}\}_{j \in J})\}_{i \in I} \rangle$ of $L$, and there are choices $\{\psi_{j}\}_{j \in J} \subset \widetilde{L}_{k}$ such that $\theta_{k}(\psi_{j})=\phi_{j}$, and
$M_{k,{\mathcal P},\{\psi_{j}\}_{j \in J}} \leq N_{1}$. Then, 
$$Z=W_{k,{\mathcal P},\{\psi_{j}\}_{j \in J}}/\widehat{N}_{2}, \quad U=W_{k,{\mathcal P},\{\psi_{j}\}_{j \in J}}/\widehat{N}_{1},\, \mbox{and} \;
{\mathcal G} \cong {\mathcal K} \rtimes L,$$
where $\widehat{N}_{j}=N_{j}/M_{k,{\mathcal P},\{\psi_{j}\}_{j \in J}}$, for $j=1,2$.
\end{enumerate}
\end{coro}

\subsection{\bf A remark}
Let $p \geq 3$ be a prime integer, $S$ be a closed Riemann surface of genus $g \geq 2$, $L \leq {\rm Aut}(S)$ isomorphic to a subgroup of index $p$ of ${\mathfrak A}_{p}$ (the alternating group in $p$ letters), and $S/L=\widehat{\mathbb C}$ of signature $(0;p,\stackrel{n+1}{\ldots},p)$, where $n \geq 2$. Let $\widetilde{\varphi}:S \to \widehat{\mathbb C}$ be a branched Galois covering with deck group $L$. Let $q \neq p$ another prime integer and let us consider the $q$-homology cover $\widetilde{S}_{q}$ together its group $M_{q} \cong {\mathbb Z}_{q}^{2g}$ of conformal automorphisms such that $S=\widetilde{S}_{q}/M_{q}$. 
If ${\mathbb Z}_{q}^{2g-1} \cong N_{1} \leq M_{q}$, then the closed Riemann surface $U=\widetilde{S}_{q}/N_{1}$ admits $\widehat{\mathcal A}=M_{q}/N_{1} \cong {\mathbb Z}_{q}$ as a group of conformal automorphisms such that $S=U/\widehat{\mathcal A}$; let $\pi_{S}:U \to S$ be an abelian Galois covering with deck group $\widehat{\mathcal A}$. 
If $N_{2}={\rm Core}_{\widetilde{L}_{q}}(N_{1})$, then the closed Riemann surface $Z=\widetilde{S}_{q}/N_{2}$ admits ${\mathcal G}=\widetilde{L}_{q}/N_{2}$ as a group of conformal automorphisms, 
such that $Z/{\mathcal G}=S/L$; let $\widetilde{\psi}:Z \to \widehat{\mathbb C}$ be a Galois covering with deck group ${\mathcal G}$.
There is also a Galois covering $\phi_{S}:Z \to S$, with deck group ${\mathcal K}=M_{q}/N_{2} \cong {\mathbb Z}_{q}^{s}$, for suitable $s \geq 0$, such that $\widetilde{\psi}=\widetilde{\varphi} \circ \phi_{S}$. The Galois closure of $\widetilde{\varphi} \circ \pi_{S}$ is given by $\widetilde{\psi}$. If $q$ is relatively prime with the order of $L$ (for instance, when $q>p$), it follows, from the Schur-Zassenhaus theorem, that ${\mathcal G} \cong {\mathcal K} \rtimes L$. If $q$ divides the order of $L$, then it might happen that ${\mathcal G}$ is not isomorphic to ${\mathcal K} \rtimes L$.
Now, assume there is a subgroup $K \leq L$, of index $p$ and acting freely on $S$ such that ${\rm Core}_{L}(K)=\{1\}$ (under this assumption, it can be seen that $L$ is necessarily a simple group). 
Let $\widetilde{K}_{q}=\theta_{q}^{-1}(K) \leq \widetilde{L}_{q}$, 
$Q:S \to X=S/K$ be a Galois covering with deck group $K$, and let $\varphi:X \to \widehat{\mathbb C}$ be a branched covering such that $\widetilde{\varphi}=\varphi \circ Q$.
Assume also that $N_{1}$ is a normal subgroup of $\widetilde{K}_{q}$. This asserts that $\widetilde{K}_{q}/N_{1}$ is a group of automorphisms of $U$ such that $U/(\widetilde{K}_{q}/N_{1}) =X$. The group $\widetilde{K}_{q}/N_{1}$ contains, as a normal subgroup, the abelian group $\widehat{\mathcal A}$, and $K \cong (\widetilde{K}_{q}/N_{1})/\widehat{\mathcal A}$.
The Galois covering $\beta=Q \circ \pi_{S}:U \to X$ has deck group $\widetilde{K}_{q}/N_{1}$. Let us also assume that $\widetilde{K}_{q}/N_{1}=\hat{K} \times \widehat{\mathcal A}$, where $\hat{K} \cong K$. Let $Y=U/\hat{K}$ and  $\pi_{Y}:U \to Y$ be a Galois covering with deck group $\hat{K}$. Let $P:Y \to X$ be such that $\beta=P \circ \pi_{Y}$. Then $P$ is a Galois covering with deck group ${\mathcal A} \cong \widehat{\mathcal A}$, and $U=Y \times_{(P,Q)} S$ (the fiber product of $P$ and $Q$).
In this particular situation, in \cite{C-R}, it is asserted that ${\mathcal G} \cong {\mathcal K} \rtimes L$ (even if $q$ is not relatively prime to the order of $L$).
\qed

\section{Example: Cyclic unbranched covers}\label{Sec:noramificado}
Let $S$ be a closed Riemann surface of genus $g \geq 2$, and $L=\langle \phi \rangle \cong {\mathbb Z}_{m}$ be a group of conformal automorphisms acting freely on $S$. By the Riemann-Hurwitz formula, $g=m\gamma+1$, for some $\gamma \geq 1$. Let us consider the presentation ${\mathcal P}=\langle \phi: \phi^{m}=1\rangle$.
There is a symplectic basis of $M_{k}$, say $a_{1},\ldots,a_{g},b_{1},\ldots,b_{g}$,
such that action of $\phi_{*}$ is given as follows:
$$\begin{array}{l}
\phi_{*}(a_{j})=a_{j+1}, \; \phi_{*}(b_{j})=b_{j+1}, \; j=1,\ldots,g-2,\\
\phi_{*}(a_{g-1})=a_{1}, \; \phi_{*}(b_{g-1})=b_{1}, \; \phi_{*}(a_{g})=a_{g}, \; \phi_{*}(b_{g})=b_{g}.
\end{array}
$$
There exists $\psi \in \widetilde{L}_{k}$ such that $\theta_{k}(\psi)=\phi$ and $\psi^{m}=b_{g}$. So, (for $k \geq 2$ and $k=\infty$)
$$\begin{array}{l}
\widetilde{L}_{k}=\langle a_{1},\ldots,a_{g},b_{1},\ldots,b_{g},\psi:
[a_{i},a_{j}]=[a_{i},b_{j}]=[b_{i},b_{j}]=1,\\
a_{1}^{k}=\cdots=a_{g}^{k}=b_{1}^{k}=\cdots=b_{g}^{k}=1, \; (i,j =1,\ldots, g); \; \psi^{m}=b_{g};\\
\psi \circ a_{j} \circ \psi^{-1}=a_{j+1},\; \psi \circ b_{j} \circ \psi^{-1}=b_{j+1},\; (j=1,\ldots,g-2);\\
\psi \circ a_{g-1} \circ \psi^{-1}=a_{1},\; \psi \circ b_{g-1} \circ \psi^{-1}=b_{1},\; \psi \circ a_{g} \circ \psi^{-1}=a_{g}, \; \psi \circ b_{g} \circ \psi^{-1}=b_{g}
\rangle,
\end{array}
$$
where, for $k=\infty$, we delete from the above the relations $a_{j}^{k}=b_{j}^{k}=1$.
In this case, 
$$\begin{array}{l}
M_{k,{\mathcal P},\{\psi\}}=\langle b_{g}\rangle \cong {\mathbb Z}_{k},\\
{\mathcal K}_{k,{\mathcal P},\{\psi\}}=M_{k}/M_{k,{\mathcal P},\{\psi\}}=\langle \hat{a}_{1},\ldots,\hat{a}_{g},\hat{b}_{1},\ldots,\hat{b}_{g-1}\rangle \cong {\mathbb Z}_{k}^{2g-1}\\
G_{k,{\mathcal P},\{\psi\}}=\widetilde{L}_{k}/M_{k,{\mathcal P},\{\psi\}}=\langle \hat{a}_{1},\ldots,\hat{a}_{g},\hat{b}_{1},\ldots,\hat{b}_{g-1},\hat{\psi}:\\
\hat{a}_{j}^{k}=\hat{b}_{i}^{m}=\hat{\psi}^{m}=1; \; 
[\hat{a}_{i},\hat{a}_{j}]=[\hat{a}_{i},\hat{b}_{j}]=[\hat{b}_{i},\hat{b}_{j}]=1;\\
\hat{\psi} \circ \hat{a}_{j} \circ \hat{\psi}^{-1}=\hat{a}_{j+1}, \; \hat{\psi} \circ \hat{b}_{j} \circ \hat{\psi}^{-1}=\hat{b}_{j+1},\; (j=1,\ldots,g-2), \; \hat{\psi} \circ \hat{a}_{g} \circ \hat{\psi}^{-1}=\hat{a}_{g} \rangle=\\ ={\mathbb Z}_{k}^{2g-1} \rtimes {\mathbb Z}_{m}.
\end{array}
$$

\begin{lemm}\label{lemita1}
\begin{enumerate}[leftmargin=*,align=left]
\item[(1)] There exists no $\rho \in \widetilde{L}_{\infty}$ of order $m$ such that $\theta_{\infty}(\rho)=\phi$.

\item[(2)] If $k \geq 2$ is finite, then there exists $\rho \in \widetilde{L}_{k}$ of order $m$ such that $\theta_{k}(\rho)=\phi$ if and only if 
$k$ and $m$ are relatively primes. 

\item[(3)] In particular, $\widetilde{L}_{k} \cong M_{k} \rtimes L$  if and only if $k$ is finite and relatively prime to $m$. 
\end{enumerate}
\end{lemm}
\begin{proof}
Every element $\rho \in \widetilde{L}_{k}$ such that $\theta_{k}(\rho)=\phi$ is of the form $\rho=\psi \circ \alpha$, for $\alpha \in M_{k}$.
If we set $\alpha_{j}=\psi^{j} \circ \alpha \circ \psi^{-j}$, for $j=1,\ldots, m-1$, then 
$$\rho^{m}=\psi \circ (\alpha \circ \alpha_{1} \circ \cdots \circ \alpha_{m} \circ b_{g}) \circ \psi^{-1}=\alpha \circ \alpha_{1} \circ \cdots \circ \alpha_{m} \circ b_{g}.$$

So, if we write 
$\alpha=a_{1}^{l_{1}} \circ \cdots \circ a_{g}^{l_{g}} \circ b_{1}^{n_{1}} \circ \cdots \circ b_{g}^{n_{g}},$
then 
$$\rho^{m}=(a_{1} \circ \cdots \circ a_{g-1})^{l_{1}+\cdots+l_{g-1}} \circ a_{g}^{ml_{g}} \circ (b_{1} \circ \cdots \circ b_{g-1})^{n_{1}+\cdots+n_{g-1}} \circ b_{g}^{mn_{g}+1}.$$
\end{proof}

\begin{example}\label{ej1}
Let us take $m=9$ and $k=6$ in the above. Any $\hat{\psi} \in \widetilde{L}_{6}$, such that $\theta_{k}(\hat{\psi})=\phi$, is of the form $\hat{\psi}=\psi \circ \alpha$, where $\alpha \in M_{6}$. Let us consider the choice 
$\hat{\psi}=\psi \circ b_{g}$. In this case, $\hat{\psi}^{9}=b_{g}^{10}=b_{g}^{4}$. So 
${\mathbb Z}_{3} \cong M_{6,{\mathcal P},\{\hat{\psi}\}}=\langle b_{g}^{2}\rangle \neq M_{6,{\mathcal P},\{\psi\}}=\langle b_{g} \rangle \cong {\mathbb Z}_{6}$, and  
$$\begin{array}{l}
{\mathcal K}_{6,{\mathcal P},\{\hat{\psi}\}}=M_{6}/M_{6,{\mathcal P},\{\hat{\psi}\}}=\langle \hat{a}_{1},\ldots,\hat{a}_{g},\hat{b}_{1},\ldots,\hat{b}_{g-1}, \hat{b}_{g}\rangle \cong {\mathbb Z}_{6}^{2g-1} \times {\mathbb Z}_{2},\\
{\mathcal G}_{6,{\mathcal P},\{\hat{\psi}\}}=\widetilde{L}_{6}/M_{6,{\mathcal P},\{\hat{\psi}\}}=\langle \hat{a}_{1},\ldots,\hat{a}_{g},\hat{b}_{1},\ldots,\hat{b}_{g-1}, \hat{b}_{g},\hat{\psi}:\\
\hat{a}_{j}^{6}=\hat{b}_{i}^{6}=\hat{b}_{g}^{2}=1; \;\hat{\psi}^{9}=\hat{b}_{g}; \; 
[\hat{a}_{i},\hat{a}_{j}]=[\hat{a}_{i},\hat{b}_{j}]=[\hat{b}_{i},\hat{b}_{j}]=1;\\
\hat{\psi} \circ \hat{a}_{j} \circ \hat{\psi}^{-1}=\hat{a}_{j+1}, \; \hat{\psi} \circ \hat{b}_{j} \circ \hat{\psi}^{-1}=\hat{b}_{j+1},\; (j=1,\ldots,g-2), \; \hat{\psi} \circ \hat{a}_{g} \circ \hat{\psi}^{-1}=\hat{a}_{g};\\
 \hat{\psi} \circ \hat{b}_{g} \circ \hat{\psi}^{-1}=\hat{b}_{g} \rangle.
\end{array}
$$
\qed
\end{example}

\begin{example}[A Galois closure example]\label{Sec:(m,k)=(2,4)}
Let $N_{1}=\langle a_{1},\ldots,a_{g},b_{2},\ldots,b_{g-1},b_{1}b_{g}^{2}\rangle \cong {\mathbb Z}_{4}^{2g-1}$, which is a non-$\psi$-invariant subgroup 
of $M_{4}$.
Note that $M_{k,{\mathcal P},\{\psi\}}=\langle b_{g} \rangle $ is not contained in $N_{1}$. So, we cannot apply Corollary \ref{coro3} to obtain a semidirect product.
In this case, $$\widehat{\mathcal A}=M_{4}/N_{1}=\langle B_{g}: B_{g}^{4}=1\rangle \cong {\mathbb Z}_{4}.$$

The maximal $\psi$-invariant subgroup of $N_{1}$ is
$$N_{2}=N_{1} \cap \psi N_{1} \psi^{-1}=\langle a_{1},\ldots,a_{g},b_{1}^{2},b_{2}^{2},b_{3},\ldots,b_{g-1},b_{1}b_{2}b_{g}^{2}\rangle,$$
which satisfies that
$$\begin{array}{l}
{\mathcal K}=M_{4}/N_{2}=\langle B_{2},B_{g}: B_{2}^{2}=B_{g}^{4}=[B_{2},B_{g}]=1\rangle \cong {\mathbb Z}_{2} \times {\mathbb Z}_{4},\\
{\mathcal U}=N_{1}/N_{2}=\langle B_{2}: B_{2}^{2}=1\rangle \cong {\mathbb Z}_{2},\\
{\mathcal G}=\widetilde{L}_{4}/N_{2}=\langle B_{2},\Psi: B_{2}^{2}=\Psi^{8}=1, \; B_{2} \circ \Psi \circ B_{2}=\Psi^{3}\rangle  \; \mbox{(the quasidihedral group of order $16$)}.
\end{array}
$$

Let us consider the Riemann surfaces
$Z=\widetilde{S}_{4}/N_{2}, \; U=\widetilde{S}_{4}/N_{1}, \; T=S/L,$
together with the Galois coverings
$\pi_{S}:U \to S$ (whose deck group is $\widehat{\mathcal A}$), $\widetilde{\varphi}:S \to T$ (whose deck group is $L$). In this situation, $S=Z/{\mathcal K}$, and 
$\widetilde{\varphi} \circ \pi_{S}:U \to T$ is a non-Galois covering of degree $8$, whose Galois closure is $\widetilde{\psi}:Z \to T$, with deck group ${\mathcal G}$, which is not a semidirect product between ${\mathcal K}$ with ${\mathbb Z}_{2}$ (since all involutions of ${\mathcal G}$ belong to ${\mathcal K}$).
\qed
\end{example}

\begin{example}\label{Carocca-Rodriguez}
Let us consider a closed Riemann surface $S$ admitting $L=\langle r,h: r^{3}=h^{2}=(r \circ h)^{5}=1\rangle \cong {\mathfrak A}_{5}$ (the alternating group of order $60$) such that $S/L$ has signature $(0;5,5,5)$. This asserts that the only elements acting with fixed points are those of order five (which are all conjugated). By the Riemann-Hurwitz formula, $S$ has genus $g=13$.
Let $\widetilde{\varphi}:S \to \widehat{\mathbb C}$ be a Galois branched covering with deck group $L$.
Let us consider the $3$-homology cover $\widetilde{S}_{3}$ and its associated group $M_{3} \cong {\mathbb Z}_{3}^{26}$. In this case, as $r \circ h$ has fixed points, the group $\widetilde{L}_{3}$ contains liftings of $r \circ h$, which are of order five (and having fixed points).  As $r$ acts freely on $S$, the Riemann surface $S/\langle r \rangle$ has genus five. So, this means that there is basis of ${\rm H}_{1}(S,{\mathbb Z})$, say $a_{1},\ldots,a_{13},b_{1},\ldots,b_{13},$
such that the homological action of $r$ is given by
$$\begin{array}{l}
r_{*}(a_{3j-2})=a_{3j-1}, \; r_{*}(a_{3j})=a_{3j-2}, \; 
r_{*}(b_{3j-2})=b_{3j-1}, \; r_{*}(b_{3j})=b_{3j-2}, \;  (j=1,\ldots, 4),\\
r_{*}(a_{13})=a_{13}, \; r_{*}(b_{13})=b_{13}.
\end{array}
$$

This means that, in $\widetilde{L}_{3}$ we may chose a lifting $\psi_{r}$ of $r$ such that $\psi_{r}^{3}=a_{13}$.
By Lemma \ref{lemita1}, there is no $\rho \in \widetilde{L}_{3}$ of order three such that $\theta_{3}(\rho)=r$. This now asserts that $\widetilde{L}_{3}$ does not contain a subgroup isomorphic to ${\mathfrak A}_{5}$ and, in particular, that it is not isomorphic to $M_{3} \rtimes {\mathfrak A}_{5}$. As $h$ has order two, we may assume that $\psi_{h} \in \widetilde{L}_{3}$ (which projects to $h$) also of order two.

Let $N_{1}=\langle a_{1},\ldots, a_{12}, b_{1}, \ldots, b_{13}\rangle \cong {\mathbb Z}_{3}^{25},$
which is $\psi_{r}$-invariant subgroup of $M_{3}$, and $N_{2}={\rm Core}_{\widetilde{L}_{3}}(N_{1}).$

Let us consider the Riemann surfaces $Z=\widetilde{S}_{3}/N_{2}$ and $U=\widetilde{S}_{3}/N_{1}$.
The inclusion of $N_{1}$ in $M_{3}$ induces a Galois unbranched covering $\pi_{S}:U \to S$, whose deck group is $\widehat{\mathcal A}=M_{3}/N_{1} \cong {\mathbb Z}_{3}$.
The inclusion of $N_{2}$ in $N_{1}$ induces a Galois unbranched covering $\pi_{U}:Z \to U$, whose deck group is the abelian group ${\mathcal U}=N_{1}/N_{2} \cong {\mathbb Z}_{3}^{s-1}$, some $s \geq 2$.
The inclusion of $N_{2}$ in $\widetilde{L}_{3}$ induces a Galois branched covering $\widetilde{\psi}:Z \to \widehat{\mathbb C}$, whose deck group is ${\mathcal G}=\widetilde{L}_{3}/N_{2}$. 
We may assume that $\widetilde{\psi}=\widetilde{\varphi} \circ \pi_{S} \circ \pi_{U}$.
By the choice of $N_{2}$, it follows that $\widetilde{\psi}$ is the Galois closure of $\widetilde{\varphi} \circ \pi_{S}$ and ${\mathcal G}$ the corresponding Galois group. 
Let us consider the subgroup ${\mathcal K}=M_{3}/N_{2} \cong {\mathbb Z}_{3}^{s}$ of ${\mathcal G}$. Note that $S=Z/{\mathcal K}$.
As $a_{13} \notin N_{1}$, it does neither belongs to $N_{2}$. This means (from the above) that ${\mathcal G}$ does not contain a copy of $L$, so it is not isomorphic to ${\mathcal K} \rtimes L$.
\qed
\end{example}

\begin{rema}
Let us consider the subgroup $K=\langle r,s=(h \circ r) \circ h \circ (h \circ r)^{-1}\rangle \cong {\mathfrak A}_{4}$ of $L$ in the above example, and let $X=S/K$. By the Riemann-Hurwitz formula, and as $K$ acts freely on $S$, the genus of $X$ is four.
Let $Q:S \to X$ be a Galois covering, with deck group $K$, and $\varphi:X \to \widehat{\mathbb C}$ be a degree five branched covering such that $\widetilde{\varphi} =\varphi \circ Q$. The Galois closure of $\varphi$ is $\widetilde{\varphi}$.
Let $\widehat{K}=\langle s,t:=r \circ s \circ r^{-1}\rangle \cong {\mathbb Z}_{2}^{2}$ and the Riemann surface $Y=S/\widehat{K}$ of genus four admitting the group ${\mathcal A}=K/\widehat{K} \cong {\mathbb Z}_{3}$ and such that $Y/{\mathcal A}=X$. Let $P:Y \to X$ be a Galois covering, with deck group ${\mathcal A}$, and $\pi_{Y}:S \to Y$ a Galois covering with deck group $\widehat{K}$, such that 
$Q=P \circ \pi_{Y}$. So, $\widetilde{\varphi}$ is also the Galois closure of $\varphi \circ P$ with deck ${\mathcal G}=L$.\qed
\end{rema}

\section{Example: Automorphisms of order three}
Let us assume $S$ a closed Riemann surface of genus $g \geq 2$, $L=\langle \phi \rangle \cong {\mathbb Z}_{3}$,  with $\phi$ acting with fixed points.

\begin{rema}[Galois closures]
As a consequence of Corollary \ref{coro3}, 
if $\widetilde{\varphi}:S \to T$ is a Galois branched covering, with deck group $L=\langle \phi \rangle \cong {\mathbb Z}_{3}$ and $\phi$ having fixed points, and $\pi_{U}:U \to S$ is an abelian unbranched covering, with deck group $\widehat{\mathcal A}$, then the Galois closure $\widetilde{\psi}:Z \to T$ of $\widetilde{\varphi} \circ \pi_{U}:U \to T$ has deck group ${\mathcal G}={\mathcal K} \rtimes {\mathbb Z}_{3}$, where ${\mathcal K} \cong {\mathbb Z}_{k}^{s}$ and $k \geq 2$ is the exponent of $\widehat{\mathcal A}$.\qed
\end{rema}

\subsection{The case $S/L$ of genus zero}
In this case, $\phi$ has $n+1$ fixed points, where $n \geq 3$, and $g=n-1$.
The Riemann orbifold $S/L$ can be uniformized by a Fuchsian group 
$$F=\langle \delta_{1},\ldots,\delta_{n+1}: \delta_{1}^{3}=\cdots=\delta_{n+1}^{3}=\delta_{1}\cdots\delta_{n+1}=1\rangle < {\rm PSL}_{2}({\mathbb R}).$$

Up to change of $\phi$ by $\phi^{-1}$, there is some $0 \leq l \leq (n+1)/2$, where $n+1-2l$ is divisible by $3$, such that the topological action of $L$ corresponds to the surjective homomorphism
$$\omega:F \to L,$$
where
$$\omega: \left\{\begin{array}{l}
\delta_{2j-1} \mapsto \phi, \; j=1,\ldots,l\\
\delta_{2j} \mapsto \phi^{-1}, \; j=1,\ldots,l\\
\delta_{2l+i} \mapsto \phi, \; i=1,\ldots,n+1-2l
\end{array}
\right\}.
$$

There is a (symplectic) basis, say $a_{1},\ldots,a_{g},b_{1},\ldots,b_{g},$
such that the action of $\phi_{*}$ is described as follows:
$$\begin{array}{l}
\phi_{*}(a_{2j-1})=a_{2j}, \; \phi_{*}(a_{2j})=a_{2j-1}^{-1}a_{2j}^{-1}, \; \phi_{*}(b_{2j-1})=b_{2j}, \; \phi_{*}(b_{2j})=b_{2j-1}^{-1}b_{2j}^{-1}, \; (j=1,\ldots,l);\\
\phi_{*}(a_{2l+i})=b_{2l+i},\; \phi_{*}(b_{2l+i})=a_{2l+i}^{-1}b_{2l+i}^{-1},\; (i=1,\ldots,g-2l).
\end{array}
$$

In this case, there is some $\psi \in \widetilde{L}_{\infty}$ such that $\theta_{\infty}(\psi)=\phi$ and $\psi^{3}=1$. Then (for $k \geq 2$ and $k=\infty$)
$$\begin{array}{l}
\widetilde{L}_{k}=\langle a_{1},\ldots,a_{g},b_{1},\ldots,b_{g},\psi:  a_{1}^{k}=\cdots=a_{g}^{k}=b_{1}^{k}=\cdots=b_{g}^{k}= \psi^{3}=1; \\
\mbox{} [a_{i},a_{j}]=[a_{i},b_{j}]=[b_{i},b_{j}]=1, \; (i,j =1,\ldots, g);\\
\psi \circ a_{2j-1} \circ \psi^{-1}=a_{2j}, \; \psi \circ a_{2j} \circ \psi^{-1}=a_{2j-1}^{-1}a_{2j}^{-1},\\
\psi \circ b_{2j-1} \circ \psi^{-1}=b_{2j}, \; \psi \circ b_{2j} \circ \psi^{-1}=b_{2j-1}^{-1}\circ b_{2j}^{-1}, \; (j=1,\ldots,l);\\
\psi \circ a_{2l+i} \psi^{-1}=b_{2l+i},\; \psi \circ b_{2l+i} \psi^{-1}=a_{2l+i}^{-1} \circ b_{2l+i}^{-1},\; (i=1,\ldots,g-2l)\rangle=\\
=\langle a_{1},\ldots,a_{g},b_{1},\ldots,b_{g}\rangle \rtimes \langle \psi \rangle \cong {\mathbb Z}_{k}^{2g} \rtimes {\mathbb Z}_{3}.
\end{array}
$$

\begin{example}[An example of Galois closure]\label{Sec:711}
Let us consider the following subgroup of $M_{k}$, which is not $\psi$-invariant:
$$N_{1}=\langle a_{1},a_{3},\ldots,a_{g},b_{1},\ldots,b_{g}\rangle \cong {\mathbb Z}_{k}^{2g-1}.$$

In this case, $$\widehat{\mathcal A}=M_{k}/N_{1}=\langle A_{2}: A_{2}^{k}=1\rangle \cong {\mathbb Z}_{k}.$$

The maximal subgroup of $N_{1}$ which is $\psi$-invariant is
$$N_{2}=N_{1} \cap \psi N_{1} \psi^{-1} \cap \psi^{-1} N_{1} \psi=\langle a_{3},\ldots,a_{g},b_{1},\ldots,b_{g}\rangle \cong {\mathbb Z}_{k}^{2g-2},$$
which satisfies that
$$\begin{array}{l}
{\mathcal K}=M_{4}/N_{2}=\langle A_{1},A_{2}: A_{1}^{k}=A_{2}^{k}=[A_{1},A_{2}]=1\rangle \cong {\mathbb Z}_{k}^{2},\\
{\mathcal U}=N_{1}/N_{2}=\langle A_{1}: A_{1}^{k}=1\rangle \cong {\mathbb Z}_{k},\\
{\mathcal G}=\widetilde{L}_{k}/N_{2}=\langle A_{1}, A_{2},\Psi: A_{1}^{k}=A_{2}^{k}=[A_{1},A_{2}]=\Psi^{3}=1, \; 
\Psi \circ A_{1} \circ \Psi^{-1}=A_{2},\\
\Psi \circ A_{2} \circ \Psi^{-1}=A_{1}^{-1} \circ A_{2}^{-1} \rangle=
\langle A_{1},A_{2}\rangle \rtimes \langle \Psi\rangle \cong {\mathbb Z}_{k}^{2} \rtimes {\mathbb Z}_{3}.
\end{array}$$

Note that, in the case $k=2$, ${\mathcal G} \cong {\mathfrak A}_{4}$ (the alternating group in four letters).

Let us consider the Riemann surfaces:
$Z=\widetilde{S}_{k}/N_{2}, \; U=\widetilde{S}_{k}/N_{1},$
together with the Galois coverings
$\pi_{U}:U \to S$ (whose deck group is ${\mathcal A}$), $\widetilde{\varphi}:S \to \widehat{\mathbb C}$ (whose deck group is $L$). Then $\widetilde{\varphi} \circ \pi_{U}:U \to \widehat{\mathbb C}$ is a non-Galois covering of degree $3k$, whose Galois closure is $\widetilde{\psi}:Z \to \widehat{\mathbb C}$, with deck group ${\mathcal G}={\mathcal K} \rtimes L \cong {\mathbb Z}_{k}^{2} \rtimes {\mathbb Z}_{3}$.
\qed
\end{example}

\subsection{The case $S/L$ of genus at least one}
In this case, $\phi$ has $n+1$ fixed points, where $n \geq 1$, and $S/L$ has genus $\gamma \geq 1$. By the Riemann-Hurwitz formula, 
$g=3\gamma+ n-1$.
The Riemann orbifold $S/L$ can be uniformized by a Fuchsian group 
$$F=\langle \alpha_{1},\beta_{1},\ldots,\alpha_{\gamma},\beta_{\gamma}, \delta_{1},\ldots,\delta_{n+1}: \delta_{1}^{3}=\cdots=\delta_{n+1}^{3}=
\delta_{1}\cdots\delta_{n+1} \prod_{j=1}^{\gamma}[\alpha_{j},\beta_{j}]=1\rangle < {\rm PSL}_{2}({\mathbb R}).$$

Up to change of $\phi$ by $\phi^{-1}$, there is some $0 \leq l \leq (n+1)/2$, where $n+1-2l$ is divisible by $3$, such that the topological action of $L$ corresponds to the surjective homomorphism
$$\omega:F \to L,$$
where
$$\omega: \left\{\begin{array}{l}
\alpha_{s}, \beta_{s} \mapsto 1, \; s=1,\ldots,\gamma\\
\delta_{2j-1} \mapsto \phi, \; j=1,\ldots,l\\
\delta_{2j} \mapsto \phi^{-1}, \; j=1,\ldots,l\\
\delta_{2l+i} \mapsto \phi, \; i=1,\ldots,n+1-2l.
\end{array}
\right.
$$

There is a (symplectic) basis, say 
$a_{1},\ldots,a_{g},b_{1},\ldots,b_{g},$
such that the action of $\phi_{*}$ is described as follows (below, we might have $l=0$):
$$\begin{array}{l}
\phi_{*}(a_{3j-2})=a_{3j-1}, \; \phi_{*}(a_{3j-1})=a_{3j}, \; \phi_{*}(a_{3j})=a_{3j-2},\\
\phi_{*}(b_{3j-2})=b_{3j-1}, \; \phi_{*}(b_{3j-1})=b_{3j}, \; \phi_{*}(b_{3j})=b_{3j-2}, \; (j=1,\ldots,\gamma);\\
\phi_{*}(a_{3\gamma+2i-1})=a_{3\gamma+2i}, \; \phi_{*}(a_{3\gamma+2i})=a_{3\gamma+2i-1}^{-1}a_{3\gamma+2i}^{-1},\\
\phi_{*}(b_{3\gamma+2i-1})=b_{3\gamma+2i}, \; \phi_{*}(b_{3\gamma+2i})=b_{3\gamma+2i-1}^{-1}b_{3\gamma+2i}^{-1}, \; (i=1,\ldots,l);\\
\phi_{*}(a_{3\gamma+2l+2l+s})=b_{3\gamma+2l+s},\; \phi_{*}(b_{3\gamma+2l+s})=a_{3\gamma+2l+s}^{-1}b_{3\gamma+2l+s}^{-1},\; (s=1,\ldots,g-3\gamma-2l).
\end{array}
$$

In this case, there is some $\psi \in \widetilde{L}_{\infty}$ such that $\theta_{\infty}(\psi)=\phi$ and $\psi^{3}=1$. Then (for $k \geq 2$ and $k=\infty$)
$$\begin{array}{l}
\widetilde{L}_{k}=\langle a_{1},\ldots,a_{g},b_{1},\ldots,b_{g},\psi: a_{1}^{k}=\cdots=a_{g}^{k}=b_{1}^{k}=\cdots=b_{g}^{k}= \psi^{3}=1;\\
\mbox{} [a_{i},a_{j}]=[a_{i},b_{j}]=[b_{i},b_{j}]=1, \; (i,j =1,\ldots, g);\\
\psi \circ a_{3j-2} \circ \psi^{-1}=a_{3j-1}, \; \psi \circ a_{3j-1} \circ \psi^{-1}=a_{3j}, \; \psi \circ a_{3j} \circ \psi^{-1}=a_{3j-2},\\
\psi \circ b_{3j-2} \circ \psi^{-1}=b_{3j-1}, \;\psi \circ b_{3j-1} \circ \psi^{-1}=b_{3j}, \; \psi \circ b_{3j} \circ \psi^{-1}=b_{3j-2}, \; (j=1,\ldots,\gamma);\\
\psi \circ a_{3\gamma+2i-1} \circ \psi^{-1}=a_{3\gamma+2i}, \; \psi \circ a_{3\gamma+2i} \circ \psi^{-1}=a_{3\gamma+2i-1}^{-1}a_{3\gamma+2i}^{-1},\\
\psi \circ b_{3\gamma+2i-1} \circ \psi^{-1}=b_{3\gamma+2i}, \; \psi \circ b_{3\gamma+2i} \circ \psi^{-1}=b_{3\gamma+2i-1}^{-1}b_{3\gamma+2i}^{-1}, \; (i=1,\ldots,l);\\
\psi \circ a_{3\gamma+2l+2l+s} \circ \psi^{-1}=b_{3\gamma+2l+s},\; \psi \circ b_{3\gamma+2l+s} \circ \psi^{-1}=a_{3\gamma+2l+s}^{-1}b_{3\gamma+2l+s}^{-1},\; (s=1,\ldots,g-3\gamma-2l) \rangle.\\
=\langle a_{1},\ldots,a_{g},b_{1},\ldots,b_{g}\rangle \rtimes \langle \psi \rangle \cong {\mathbb Z}_{k}^{2g} \rtimes {\mathbb Z}_{3}.
\end{array}
$$

\subsection{A particular case: $\gamma=1$ and $n=5$}
In this case, $S$ has genus $g=7$, $\phi$ has exactly $6$ fixed points, and $T=S/L$ is a Riemann orbifold of genus one with exactly $6$ cone points of order three. The Fuchsian group $F$, in this case, is  
$$F=\langle \alpha,\beta,\delta_{1},\ldots,\delta_{6}: [\alpha,\beta]\delta_{1}\cdots \delta_{6}=1=\delta_{1}^{3}=\cdots=\delta_{6}^{3}\rangle < {\rm PSL}_{2}({\mathbb R}).$$

The two topologically different actions of $L$ correspond to the homomorphisms
$$\omega_{1},\omega_{2}:F \to L,$$
where
$$\omega_{1}: \left\{\begin{array}{l}
\alpha, \beta \mapsto 1\\
\delta_{1},\ldots,\delta_{6} \mapsto \phi
\end{array}
\right\};\;
\omega_{2}: \left\{\begin{array}{l}
\alpha, \beta \mapsto 1\\
\delta_{1},\delta_{3},\delta_{5} \mapsto \phi\\
\delta_{2},\delta_{4},\delta_{6} \mapsto \phi^{-1}
\end{array}
\right\}.
$$

\subsubsection{\bf The case of $\omega_{1}$}
There is a (symplectic) basis, say 
$a_{1},\ldots,a_{7},b_{1},\ldots,b_{7},$
such that the actions of $\phi_{*}$ is described as follows:
$$\begin{array}{l}
\phi_{*}(a_{1})=a_{2}, \; \phi_{*}(a_{2})=a_{3}, \; \phi_{*}(a_{3})=a_{1},\; \phi_{*}(b_{1})=b_{2}, \; \phi_{*}(b_{2})=b_{3}, \; \phi_{*}(b_{3})=b_{1},\\
\phi_{*}(a_{4})=a_{5}, \; \phi_{*}(a_{5})=a_{4}^{-1}\circ a_{5}^{-1}, \; \phi_{*}(b_{4})=b_{5}, \; \phi_{*}(b_{5})=b_{4}^{-1}\circ b_{5}^{-1},\\
\phi_{*}(a_{6})=b_{6}, \; \phi_{*}(b_{6})=a_{6}^{-1} \circ b_{6}^{-1}, \; \phi_{*}(a_{7})=b_{7}, \; \phi_{*}(b_{7})=a_{7}^{-1} \circ b_{7}^{-1}.
\end{array}
$$

In this case,
$$\begin{array}{l}
\widetilde{L}_{k}=\langle a_{1},\ldots,a_{7},b_{1},\ldots,b_{7},\psi:\\
a_{1}^{k}=\cdots=a_{7}^{k}=b_{1}^{k}=\cdots=b_{7}^{k}=\psi^{3}=1, \; [a_{i},a_{j}]=[a_{i},b_{j}]=[b_{i},b_{j}]=1,\\
\psi \circ a_{1} \circ \psi^{-1}=a_{2}, \; \psi \circ a_{2} \circ \psi^{-1}=a_{3}, \; \psi \circ a_{3} \circ \psi^{-1}=a_{1},\\
\psi \circ b_{1} \circ \psi^{-1}=b_{2}, \; \psi \circ b_{2} \psi^{-1}=b_{3}, \; \psi \circ b_{3} \circ \psi^{-1}=b_{1},\\
\psi \circ a_{4} \circ \psi^{-1}=a_{5}, \; \psi \circ a_{5} \circ \psi^{-1}=a_{4}^{-1}\circ a_{5}^{-1},\;
\psi \circ b_{4} \circ \psi^{-1}=b_{5}, \; \psi \circ b_{5} \circ \psi^{-1}=b_{4}^{-1}\circ b_{5}^{-1},\\
$$\psi \circ a_{6} \circ \psi^{-1}=b_{6}, \; \psi \circ b_{6} \circ \psi^{-1}=a_{6}^{-1} \circ b_{6}^{-1}, \;
\psi \circ a_{7} \circ \psi^{-1}=b_{7}, \; \psi \circ b_{7} \circ \psi^{-1}=a_{7}^{-1} \circ b_{7}^{-1}\rangle=\\
=M_{k} \rtimes \langle \psi\rangle \cong {\mathbb Z}_{k}^{14} \rtimes {\mathbb Z}_{3}.
\end{array}
$$

\subsubsection{\bf The case of $\omega_{2}$}
There is a (symplectic) basis, say  
$a_{1},\ldots,a_{7},b_{1},\ldots,b_{7},$
such that the actions of $\phi_{*}$ is described as follows:
$$\begin{array}{l}
\phi_{*}(a_{1})=a_{2}, \; \phi_{*}(a_{2})=a_{3}, \; \phi_{*}(a_{3})=a_{1},\;
\phi_{*}(b_{1})=b_{2}, \; \phi_{*}(b_{2})=b_{3}, \; \phi_{*}(b_{3})=b_{1},\\
\phi_{*}(a_{4})=a_{5}, \; \phi_{*}(a_{5})=a_{4}^{-1}\circ a_{5}^{-1},\;
\phi_{*}(b_{4})=b_{5}, \; \phi_{*}(b_{5})=b_{4}^{-1}\circ b_{5}^{-1},\\
\phi_{*}(a_{6})=a_{7}, \; \phi_{*}(a_{7})=a_{6}^{-1} \circ a_{7}^{-1}, \;
\phi_{*}(b_{6})=b_{7}, \; \phi_{*}(b_{7})=b_{6}^{-1} \circ b_{7}^{-1}.
\end{array}
$$

In this case,
$$\begin{array}{l}
\widetilde{L}_{k}=\langle a_{1},\ldots,a_{7},b_{1},\ldots,b_{7},\psi:\\
a_{1}^{k}=\cdots=a_{7}^{k}=b_{1}^{k}=\cdots=b_{7}^{k}=\psi^{3}=1, \; [a_{i},a_{j}]=[a_{i},b_{j}]=[b_{i},b_{j}]=1,\\
\psi \circ a_{1} \circ \psi^{-1}=a_{2}, \; \psi \circ a_{2} \circ \psi^{-1}=a_{3}, \; \psi \circ a_{3} \circ \psi^{-1}=a_{1},\\
\psi \circ b_{1} \circ \psi^{-1}=b_{2}, \; \psi \circ b_{2} \psi^{-1}=b_{3}, \; \psi \circ b_{3} \circ \psi^{-1}=b_{1},\\
\psi \circ a_{4} \circ \psi^{-1}=a_{5}, \; \psi \circ a_{5} \circ \psi^{-1}=a_{4}^{-1}\circ a_{5}^{-1},\;
\psi \circ b_{4} \circ \psi^{-1}=b_{5}, \; \psi \circ b_{5} \circ \psi^{-1}=b_{4}^{-1}\circ b_{5}^{-1},\\
\psi \circ a_{6} \circ \psi^{-1}=a_{7}, \; \psi \circ a_{7} \circ \psi^{-1}=a_{6}^{-1} \circ a_{7}^{-1}, \;
\psi \circ b_{6} \circ \psi^{-1}=b_{7}, \; \psi \circ b_{7} \circ \psi^{-1}=b_{6}^{-1} \circ b_{7}^{-1}\rangle=\\
=M_{k} \rtimes \langle \psi\rangle \cong {\mathbb Z}_{k}^{14} \rtimes {\mathbb Z}_{3}.
\end{array}
$$

\begin{example}[$k=3$]
Let us consider the following subgroup of $M_{3}$, which is $\psi$-invariant
$$N_{1}=\langle a_{1} \circ a_{2}^{-1},a_{1} \circ a_{3}^{-1},a_{4},\ldots,a_{7},b_{1},\ldots,b_{7}\rangle \cong {\mathbb Z}_{3}^{13}.$$

Let us consider the Riemann surface $Z=\widetilde{S}_{3}/N_{1}$ of genus $19$.
This surface admits the group
$${\mathcal G}=\widetilde{L}_{3}/N_{1}=\langle A_{1}, \Psi: A_{1}^{3}=\Psi^{3}=[A_{1},\Psi]=1 \rangle \cong {\mathbb Z}_{3}^{2}$$
as a group of conformal automorphisms, such that $S=Z/\langle A_{1} \rangle$, $Z/{\mathcal G}=S/T$, and so that the only elements of ${\mathcal G}$ acting with fixed points on $Z$ are the powers of $\Psi$.
\qed
\end{example}

\section{Example: The symmetric group ${\mathfrak S}_{3}$}\label{Sec:S3}
We now assume $L \cong {\mathfrak S}_{3}$ with the presentation ${\mathcal P}=\langle r,h: r^{3}=h^{2}=(r \circ h)^{2}=1\rangle$, $S$ a closed Riemann surface of genus $g \geq 2$, 
${\rm Fix}(r)=\emptyset$, and ${\rm Fix}(h) \neq \emptyset$ (so $h$, $r \circ h$ and $r^{2} \circ h$ have fixed points; these are permuted by $r$). 
Let $n \geq 1$ (necessarily odd) be such that $h$ has $n+1$ fixed points. By the Riemann-Hurwitz formula, $g=3 \gamma+\frac{3n-7}{2}$, where $\gamma$ is the genus of $S/L$.  As we are assuming $g \geq 2$, then, for $\gamma=0$, we must have that $n \geq 5$ is odd.

To simplify notations, we will assume that  $S/L$ of genus zero (but everything can be stated for the general case).

A Fuchsian group uniformizing the Riemann orbifold $S/L$ has presentation
$$F=\langle \delta_{1},\ldots,\delta_{n+1}: \delta_{1}\cdots \delta_{n+1}=1=\delta_{1}^{2}=\cdots=\delta_{n+1}^{2}\rangle < {\rm PSL}_{2}({\mathbb R}).$$

The topological action of $L$ corresponds to a homomorphism
$$\omega_{l_{1},l_{2}}:F \to L,$$
where
$$\omega_{l_{1},l_{2}}: \left\{\begin{array}{l}
\delta_{1},\ldots,\delta_{2l_{1}} \mapsto h\\
\delta_{2l_{1}+1},\ldots,\delta_{2l_{1}+2l_{2}} \mapsto r\circ h\\
\delta_{2l_{1}+2l_{2}+1},\ldots,\delta_{n+1} \mapsto r^{2} \circ h.
\end{array}
\right.
$$

They are topologically equivalent to $\omega_{n-1,2}$.

There is a (symplectic) basis, say  
$a_{1},\ldots,a_{g},b_{1},\ldots,b_{g},$
such that the actions of $r_{*}$ and $h_{8}$ are described as follows:
$$\begin{array}{l}
r_{*}(a_{3j-2})=a_{3j-1}, \; r_{*}(a_{3j-1})=a_{3j}, \;
r_{*}(b_{3j-2})=b_{3j-1}, \; r_{*}(b_{3j-1})=b_{3j}, \; (j=1,\ldots,\frac{g-1}{3}),\\
r_{*}(a_{g})=a_{g}, \; r_{*}(b_{g})=b_{g},\\
h_{*}(a_{3j-2})=a_{3j-2}^{-1}, \; h_{*}(a_{3j-1}=a_{3j}^{-1}, \; 
h_{*}(b_{3j-2})=b_{3j-2}^{-1}, \; h_{*}(b_{3j-1})=b_{3j}^{-1}, \; (j=1,\ldots,\frac{g-1}{3}),\\
h_{*}(a_{g})=a_{g}^{-1}, \; h_{*}(b_{g})=b_{g}^{-1}.
\end{array}
$$

In this case, there are $\psi_{r},\psi_{h} \in \widetilde{L}_{\infty}$ such that
$$\theta_{\infty}(\psi_{r})=r, \; \theta_{\infty}(\psi_{h})=h,\;
\psi_{h}^{2}=(\psi_{r} \circ \psi_{h})^{2}=1, \; \psi_{r}^{3}=b_{g}.$$

In particular,
$$\begin{array}{l}
\widetilde{L}_{k}=\langle a_{1},\ldots,a_{g},b_{1},\ldots,b_{g},\psi_{r},\psi_{h}:\\
a_{1}^{k}=\cdots=a_{g}^{k}=b_{1}^{k}=\cdots=b_{g}^{k}=1;\; 
\psi_{r}^{3}=b_{g}, \; \psi_{h}^{2}=(\psi_{r} \circ \psi_{h})^{2}=1; \\
\mbox{}[a_{i},a_{j}]=[a_{i},b_{j}]=[b_{i},b_{j}]=1, \; (i,j =1,\ldots, g);\\
\psi_{r} \circ a_{3j-2} \circ \psi_{r}^{-1}=a_{3j-1}, \; \psi_{r} \circ a_{3j-1} \circ \psi_{r}^{-1}=a_{3j},\\
\psi_{r} \circ b_{3j-2} \circ \psi_{r}^{-1}=b_{3j-1}, \; \psi_{r} \circ b_{3j-1} \circ \psi_{r}^{-1}=b_{3j}, \; (j=1,\ldots,\frac{g-1}{3}),\\
\psi_{r} \circ a_{g} \circ \psi_{r}^{-1}=a_{g}, \; \psi_{r} \circ b_{g} \circ \psi_{r}^{-1}=b_{g},\;
\psi_{h} \circ a_{3j-2} \circ \psi_{h}=a_{3j-2}^{-1}, \; \psi_{h} \circ a_{3j-1} \circ \psi_{h}=a_{3j}^{-1},\\
\psi_{h} \circ b_{3j-2} \circ \psi_{h}=b_{3j-2}^{-1}, \; \psi_{h} \circ b_{3j-1} \circ \psi_{h}=b_{3j}^{-1}, \; (j=1,\ldots,\frac{g-1}{3}),\\
\psi_{h} \circ a_{g} \circ \psi_{h}=a_{g}^{-1}, \; \psi_{h} \circ b_{g} \circ \psi_{h}=b_{g}^{-1}\rangle,\\
M_{k,{\mathcal P},\{\psi_{r},\psi_{h}\}}=\langle b_{g} \rangle,\\
{\mathcal K}_{k,{\mathcal P},\{\psi_{r},\psi_{h}\}}=M_{k}/M_{k,{\mathcal P},\{\psi_{r},\psi_{h}\}}=\langle a_{1},\ldots,a_{g},b_{1},\ldots,b_{g-1}\rangle \cong {\mathbb Z}_{k}^{2g-1},\\ 
{\mathcal G}_{k,{\mathcal P},\{\psi_{r},\psi_{h}\}}=A_{k,{\mathcal P},\{\psi_{r},\psi_{h}\}} \rtimes \langle \Psi_{r}, \Psi_{h} \rangle \cong {\mathbb Z}_{k}^{2g-1} \rtimes {\mathfrak S}_{3}.
\end{array}
$$

\subsection{Case $k \geq 2$ is not divisible by $3$}
 
\begin{lemm}\label{lema1}
If $k \geq 2$ is not divisible by $3$, then there exists $\rho \in \widetilde{L}_{k}$ of order three with $\theta_{k}(\rho)=r$ such that $(\rho \circ \psi_{h})^{2}=1$. 
In particular, 
$$\widetilde{L}_{k}=M_{k} \rtimes \langle \rho, \psi_{h}\rangle \cong {\mathbb Z}_{k}^{2g} \rtimes {\mathfrak S}_{3}.$$
\end{lemm}
\begin{proof}
There are integers $l,s \neq 0$ such that $lk=1+3s$. If we set $\rho=\psi_{r}^{lk}=\psi_{r} \circ b_{g}^{s}$, then 
$\theta_{k}(\rho)=r$, and 
$(\rho \circ \psi_{h})^{2}=\rho \circ \psi_{h} \circ \rho \circ \psi_{h}=\psi_{r}^{lk} \circ \psi_{h} \circ \psi_{r}^{lk} \circ \psi_{h}=
\psi_{r} \circ (b_{g}^{s} \circ \psi_{h} \circ b_{g}^{s}) \circ \psi_{r}\circ \psi_{h}=
\psi_{r} \circ \psi_{h} \circ \psi_{r}\circ \psi_{h}=1$.
\end{proof}

\begin{rema}[Galois closures]
As a consequence, 
if $\widetilde{\varphi}:S \to T$ is a Galois branched covering, with deck group $L\cong {\mathfrak S}_{3}$ such that the elements of order two have fixed points and the one of order three does not,  and $\pi_{U}:U \to S$ is an abelian unbranched covering, with deck group $\widehat{\mathcal A}$, then the Galois closure $\widetilde{\psi}:Z \to T=S/L$ of $\widetilde{\varphi} \circ \pi_{U}:U \to T$ has deck group ${\mathcal G}={\mathcal K} \rtimes {\mathfrak S}_{3}$, where ${\mathcal K}$ is isomorphic to a subgroup of ${\mathbb Z}_{k}^{2g}$ and $k \geq 2$ is the exponent of $\widehat{\mathcal A}$.\qed
\end{rema}

\begin{example}[$k=2$]\label{Sec:k=2}
By Lemma \ref{lema1}, 
$$\begin{array}{l}
\widetilde{L}_{2}=\langle a_{1},\ldots,a_{g},b_{1},\ldots,b_{g},\rho,\psi_{h}:
a_{1}^{2}=\cdots=a_{g}^{2}=b_{1}^{2}=\cdots=b_{g}^{2}=1;\\
\rho^{3}=\psi_{h}^{2}=(\rho \circ \psi_{h})^{2}=1; \;
[a_{i},a_{j}]=[a_{i},b_{j}]=[b_{i},b_{j}]=1, \; (i,j =1,\ldots, g);\\
\rho \circ a_{3j-2} \circ \rho^{-1}=a_{3j-1}, \; \rho \circ a_{3j-1} \circ \rho^{-1}=a_{3j},\\
\rho \circ b_{3j-2} \circ \rho^{-1}=b_{3j-1}, \; \rho \circ b_{3j-1} \circ \rho^{-1}=b_{3j}, \; (j=1,\ldots,\frac{g-1}{3}),\\
\rho \circ a_{g} \circ \rho^{-1}=a_{g}, \; \rho \circ b_{g} \circ \rho^{-1}=b_{g},\;
\psi_{h} \circ a_{3j-2} \circ \psi_{h}=a_{3j-2}^{-1}, \; \psi_{h} \circ a_{3j-1} \circ \psi_{h}=a_{3j}^{-1},\\
\psi_{h} \circ b_{3j-2} \circ \psi_{h}=b_{3j-2}^{-1}, \; \psi_{h} \circ b_{3j-1} \circ \psi_{h}=b_{3j}^{-1}, \; (j=1,\ldots,\frac{g-1}{3}),\\
\psi_{h} \circ a_{g} \psi_{h}=a_{g}^{-1}, \; \psi_{h} \circ b_{g} \psi_{h}=b_{g}^{-1}\rangle=
\langle a_{1},\ldots,a_{g},b_{1},\ldots,b_{g}\rangle \rtimes \langle \rho,\psi_{h} \rangle \cong {\mathbb Z}_{2}^{2g} \rtimes {\mathfrak S}_{3}.
\end{array}
$$


Let us consider the following $\psi_{h}$-invariant subgroup of $M_{2}$ (which is not $\rho$-invariant)
$$N_{1}=\langle a_{1},a_{2}a_{3},a_{4},\ldots,a_{g},b_{1},\ldots,b_{g}\rangle \cong {\mathbb Z}_{2}^{2g-1}.$$

The maximal $\rho$-invariant subgroup of $N_{1}$ is
$$N_{2}=N_{1} \cap \rho N_{1} \rho^{-1} \cap \rho^{-1} N_{1} \rho=\langle a_{1}a_{2}a_{3},a_{4},\ldots,a_{g},b_{1},\ldots,b_{g}\rangle \cong {\mathbb Z}_{2}^{2g-2},$$

In this case, where $\theta_{2}(a_{j})=A_{j}$, $\theta_{2}(b_{j})=B_{j}$, $\theta_{2}(\rho)=\Psi_{r}$, and $\theta_{2}(\psi_{h})=\Psi_{h}$,
$$\begin{array}{l}
{\mathcal U}=N_{1}/N_{2}=\langle A_{1}: A_{1}^{2}=1\rangle \cong {\mathbb Z}_{2},\\
{\mathcal G}=\widetilde{L}_{2}/N_{2}=\langle A_{1}, A_{2},\Psi_{r}, \Psi_{h}: 
A_{1}^{2}=A_{2}^{2}=[A_{1},A_{2}]=\Psi_{r}^{3}=\Psi_{h}^{2}=(\Psi_{r} \circ \Psi_{h})^{2}=1,\\
\Psi_{r} \circ A_{1} \circ \Psi_{r}^{-1}=A_{2},  \Psi_{r} \circ A_{2} \circ \Psi_{r}^{-1}=A_{1} \circ A_{2},\\
\Psi_{h} \circ A_{1} \circ \Psi_{h}=A_{1},  \Psi_{h} \circ A_{2} \circ \Psi_{h}=A_{1} \circ A_{2},
\rangle \cong {\mathfrak S}_{4},\\
\widehat{L}=\widetilde{L}_{2}/M_{2}=\langle \Psi_{r}, \Psi_{h}: \Psi_{r}^{3}=\Psi_{h}^{2}=(\Psi_{r} \circ \Psi_{h})^{2}=1 \cong {\mathfrak S}_{3}.
\end{array}
$$

Set
$$\begin{array}{l}
\widehat{\mathcal A}=M_{2}/N_{1} \cong {\mathbb Z}_{2},\\
{\mathcal A}=\langle M_{2}, \psi_{h}\rangle /\langle N_{1}, \psi_{h}\rangle=M_{2}/N_{1}=\langle A_{2}: A_{2}^{2}=1\rangle \cong {\mathbb Z}_{2},\\
{\mathcal H}=\langle M_{2}, \psi_{h}\rangle/N_{2}=\langle A_{1},A_{2}, \Psi_{h}: 
A_{1}^{2}=A_{2}^{2}=[A_{1},A_{2}]=\Psi_{h}^{2}=1,\\
\Psi_{h} \circ A_{1} \circ \Psi_{h}=A_{1},  \Psi_{h} \circ A_{2} \circ \Psi_{h}=A_{1} \circ A_{2}\rangle=\langle A_{1}, A_{2} \rangle \rtimes \langle \Psi_{h}\rangle  \cong D_{4}\\
{\mathcal K}=M_{2}/N_{2}=\langle A_{1},A_{2}: A_{1}^{2}=A_{2}^{2}=[A_{1},A_{2}]=1\rangle \cong {\mathbb Z}_{2}^{2},
\end{array}
$$
and the Riemann surfaces
$Z=\widetilde{S}_{2}/N_{2}, \; Y=\widetilde{S}_{2}/\langle N_{1}, \psi_{h}\rangle, \; X=\widetilde{S}_{2}/\langle M_{2}, \psi_{h}\rangle,$
together with the Galois coverings
$P:Y \to X$ (whose deck group is ${\mathcal A}$), $\varphi:X \to \widehat{\mathbb C}$ a covering induced by the inclusion of $\langle M_{2},\psi_{h}\rangle$ in ${\mathcal G}$,  
$\widetilde{\varphi}:S \to \widehat{\mathbb C}$ whose deck group is $L$, and $\pi_{U}:U \to S$ a Galois covering with deck group $\widehat{\mathcal A}$. 

Note that $\varphi \circ P:Y \to \widehat{\mathbb C}$ is a non-Galois covering of degree $6$, and that the Galois branched covering $\widetilde{\varphi}$ is the Galois closure of $\varphi$. Moreover, both $\widetilde{\varphi} \circ \pi_{S}$ and $\varphi \circ P$ have the same Galois closure $\widetilde{\psi}:Z \to \widehat{\mathbb C}$, with deck group ${\mathcal G}={\mathcal K} \rtimes \widehat{L} \cong {\mathbb Z}_{2}^{2} \rtimes {\mathfrak S}_{3} \cong {\mathfrak S}_{4}$.
\qed
\end{example}

\subsection{Case $k \geq 3$ is divisible by $3$}
\begin{lemm}\label{lema2}
If $k \geq 3$ is divisible by $3$, then there is no $\rho \in \widetilde{L}_{k}$ such that $\rho^{3}=1$ and $\theta_{k}(\rho)=r$. In particular, 
$\widetilde{L}_{k}$ is not isomorphic to $M_{k} \rtimes {\mathfrak S}_{3}$.
\end{lemm}
\begin{proof}
Every element $ \rho \in \widetilde{L}_{k}$ such that $\theta_{k}(\rho)=r$ has the form $\psi_{r} \circ m$, where $m \in M_{k}$.
Let us write 
$$m=a_{1}^{\alpha_{1}} \circ \cdots \circ a_{g}^{\alpha_{g}} \circ b_{1}^{\beta_{1}} \circ \cdots \circ b_{g}^{\beta_{g}},$$ 
then 
$$(\psi_{r} \circ m)^{3}=\hat{m} \circ \hat{\hat{m}} \circ \psi_{r}^{3} \circ m=\hat{m} \circ \hat{\hat{m}} \circ b_{g} \circ m,$$
where 
$$\hat{m}=\psi_{r} \circ m \circ \psi_{r}^{-1}, \; \hat{\hat{m}}=\psi_{r} \circ \hat{m} \circ \psi_{r}^{-1}.$$

In order for $\rho^{3}=1$, we must have that $m \circ \hat{m} \circ \hat{\hat{m}} \circ b_{g}=1$. But, the exponent of $b_{g}$ in 
$m \circ \hat{m} \circ \hat{\hat{m}} \circ b_{g}$ is $1+3\beta_{g}$, a contradiction.
\end{proof}

\begin{rema}
Note that the above lemma only asserts that the group $\widetilde{L}_{3}$ does not contain a copy of ${\mathfrak S}_{3}$ that projects onto $L$. It might have other subgroups isomorphic to ${\mathfrak S}_{3}$, but they are projected to the subgroup $\langle h \rangle$ of $L$.\qed
\end{rema}

\begin{example}[$n=5$ and $k=3$]\label{Sec:(n,k)=(5,3)}
In this case, $S$ has genus $g=4$, and $S/L$ is the Riemann sphere with exactly $6$ conical points, each one of order $2$. Let $\widetilde{\varphi}:S \to \widehat{\mathbb C}$ be a branched Galois covering with deck group $L$.
Let $N_{1}$ be a $\psi_{h}$-invariant subgroup of $M_{3}$, $U=\widetilde{S}_{3}/N_{1}$ and $\pi_{S}:U \to S$
an abelian unbranched Galois covering  with deck group $\widehat{\mathcal A}=M_{3}/N_{1}$. If $N_{2}={\rm Core}_{\widetilde{L}_{3}}(N_{1})$, then we also have the Riemann surface $Z=\widetilde{S}_{3}/N_{2}$, together an abelian unbranched covering $\pi_{U}:Z \to U$ with deck group ${\mathcal U}=N_{1}/N_{2}$. If ${\mathcal K}=M_{3}/N_{2}$, then $S=Z/{\mathcal K}$.
The Galois closure of $\widetilde{\varphi} \circ \pi_{S}$ is $\widetilde{\psi}=\widetilde{\varphi} \circ \pi_{S} \circ \pi_{U}:Z \to \widehat{\mathbb C}$, with deck group ${\mathcal G}=\widetilde{L}_{3}/N_{2}$. 
As a consequence of Lemma \ref{lema2}, 
${\mathcal G}={\mathcal K} \rtimes \hat{L} \cong {\mathcal K}\rtimes L$ if and only if $b_{4} \in N_{2}$ (which is equivalent to have $b_{4} \in N_{1}$ as $\langle b_{4}\rangle$ is normal subgroup of $\widetilde{L}_{3}$). Below, we consider some examples for different choices of subgroups $N_{1}$.

\subsubsection*{Case (1): $N_{1}=N_{2} \cong {\mathbb Z}_{3}^{7}$}
In this case,  ${\mathcal K}={\mathbb Z}_{3}$.
With the help of GAP \cite{GAP}, we can observe exactly $40$ possibilities for $N_{1}$. Only $13$ of them have the property that $b_{4} \in N_{1}$, so ${\mathcal G} ={\mathcal K} \rtimes \hat{L} \cong {\mathbb Z}_{3} \rtimes {\mathfrak S}_{3}$. In the other $27$ cases, this semidirect product is not true (but it must observe that inside ${\mathcal G}$ there are subgroups isomorphic to ${\mathfrak S}_{3}$, all of them are projected to $\langle h \rangle$).

\subsubsection*{Case (2): $N_{1}=N_{2} \cong {\mathbb Z}_{3}^{6}$}
Take $N_{1}=\langle a_{1}a_{2}a_{3},a_{4},b_{1},b_{2},b_{3},b_{4}\rangle \cong {\mathbb Z}_{3}^{6}$, which is normal subgroup of $\widetilde{L}_{3}$, and $b_{4} \in N_{1}$. In this case, 
${\mathcal K}={\mathbb Z}_{3}^{2}$, and
$$\begin{array}{l}
{\mathcal G}=\widetilde{L}_{3}/N_{2}=\langle A_{1},A_{2},\Psi_{r},\Psi_{h}: A_{1}^{3}=A_{2}^{3}=\Psi_{r}^{3}=\Psi_{h}^{2}=(\Psi_{r} \circ \Psi_{h})^{2}=[A_{1},A_{2}]=1;\\
\Psi_{r} \circ A_{1} \circ \Psi_{r}^{-1}=A_{2}; \; \Psi_{r} \circ A_{2} \circ \Psi_{r}^{-1}=A_{1}^{-1} \circ A_{2}^{-1};\;
\Psi_{h} \circ A_{1} \circ \Psi_{h}=A_{1}^{-1}; \\ \Psi_{h} \circ A_{2} \circ \Psi_{h}=A_{1} \circ A_{2} \rangle=
\langle A_{1}, A_{2}\rangle \rtimes \langle \Psi_{r},\Psi_{h}\rangle \cong {\mathbb Z}_{3}^{2} \rtimes {\mathfrak S }_{3}
\end{array}
$$

\subsubsection*{Case (3): $N_{1}=N_{2} \cong {\mathbb Z}_{3}^{6}$}
Take $N_{1}=\langle a_{1},a_{2},a_{3},b_{1},b_{2},b_{3}\rangle \cong {\mathbb Z}_{3}^{6}$, which normal subgroup of $\widetilde{L}_{3}$. In this case, 
${\mathcal K}={\mathbb Z}_{3}^{2}$. As in this case $b_{4} \notin N_{1}$,  ${\mathcal G}$ cannot be the semidirect product ${\mathcal K} \rtimes \hat{L} \cong {\mathbb Z}_{3}^{2} \rtimes {\mathfrak S }_{3}$, in fact 
$$\begin{array}{l}
{\mathcal G}=\widetilde{L}_{3}/N_{2}=\langle A_{4},B_{4},\Psi_{r},\Psi_{h}: A_{4}^{3}=B_{4}^{3}=\Psi_{h}^{2}=(\Psi_{r} \circ \Psi_{h})^{2}=[A_{4},B_{4}]=1; \;  \Psi_{r}^{3}=B_{4};\\
\Psi_{r} \circ A_{4} \circ \Psi_{r}^{-1}=A_{4}; \; \Psi_{r} \circ B_{4} \circ \Psi_{r}^{-1}=B_{4};\;
\Psi_{h} \circ A_{4} \circ \Psi_{h}=A_{4}^{-1}; \; \Psi_{h} \circ B_{4} \circ \Psi_{h}=B_{4}^{-1} \rangle.
\end{array}
$$

\subsubsection*{Case (4): $N_{1} \cong {\mathbb Z}_{3}^{6}$, $N_{1} \neq N_{2}$}
Take $N_{1}=\langle a_{1},a_{2},a_{4},b_{1},b_{2},b_{3}\rangle \cong {\mathbb Z}_{3}^{6}$, which is $\psi_{h}$-invariant, but it is not $\psi_{r}$-invariant. Its maximal subgroup which is $\langle \psi_{h}, \psi_{r}\rangle$-invariant is $N_{2}=\langle a_{4},b_{1},b_{2},b_{3}\rangle \cong {\mathbb Z}_{3}^{4}$. Note that, as in this case $b_{4} \notin N_{1}$,  ${\mathcal G}$ cannot be the semidirect product between ${\mathcal K} = M_{3}/N_{2}\cong {\mathbb Z}_{3}^{2}$ with a subgroup $\hat{L} \cong {\mathfrak S }_{3}$ of ${\mathcal G}$. In this case, 
$$\begin{array}{l}
{\mathcal G}=\widetilde{L}_{3}/N_{2}=\langle A_{1},A_{2},A_{3},B_{4},\Psi_{r},\Psi_{h}: A_{j}^{3}=B_{4}^{3}=\Psi_{h}^{2}=(\Psi_{r} \circ \Psi_{h})^{2}=1; \; \Psi_{r}^{3}=B_{4};\\
\mbox{}[A_{i},A_{j}]=[A_{i},B_{4}]=1 \; (i,j =1,2,3); \\
\Psi_{r} \circ A_{1} \circ \Psi_{r}^{-1}=A_{2}; \; \Psi_{r} \circ A_{2} \circ \Psi_{r}^{-1}=A_{3}; \; \Psi_{r} \circ A_{3} \circ \Psi_{r}^{-1}=A_{1};\; \Psi_{r} \circ B_{4} \circ \Psi_{r}^{-1}=B_{4};\\
\Psi_{h} \circ A_{1} \circ \Psi_{h}=A_{1}^{-1}; \; \Psi_{h} \circ A_{2} \circ \Psi_{h}=A_{3}^{-1}; \; \Psi_{h} \circ A_{3} \circ \Psi_{h}=A_{2}^{-1};\; \Psi_{h} \circ B_{4} \circ \Psi_{h}=B_{4}^{-1} \rangle.
\end{array}
$$

We may note that ${\mathcal G}$ contains many copies of ${\mathfrak S}_{3}$, but each of them is projected to $\langle h \rangle$.

\subsubsection*{Case (5): $N_{1} \cong {\mathbb Z}_{3}^{6}$, $N_{1} \neq N_{2}$}
Take $N_{1}=\langle a_{1},a_{2},a_{4},b_{1},b_{2},b_{4}\rangle \cong {\mathbb Z}_{3}^{6}$, which is $\psi_{h}$-invariant, but it is not $\psi_{r}$-invariant. Its maximal subgroup which is $\langle \psi_{h}, \psi_{r}\rangle$-invariant is $N_{2}=\langle a_{4},b_{4}\rangle \cong {\mathbb Z}_{3}^{2}$, and 
${\mathcal K}={\mathbb Z}_{3}^{6}$, and 
$$\begin{array}{l}
{\mathcal G}=\widetilde{L}_{3}/N_{2}=\langle A_{1},A_{2},A_{3},B_{1},B_{2},b_{3},\Psi_{r},\Psi_{h}: A_{j}^{3}=B_{4}^{3}=\Psi_{r}^{3}=\Psi_{h}^{2}=(\Psi_{r} \circ \Psi_{h})^{2}=1; \\
\mbox{}[A_{i},A_{j}]=[A_{i},B_{4}]=1 \; (i,j =1,2,3); \\
\Psi_{r} \circ A_{1} \circ \Psi_{r}^{-1}=A_{2}; \; \Psi_{r} \circ A_{2} \circ \Psi_{r}^{-1}=A_{3}; \; \Psi_{r} \circ A_{3} \circ \Psi_{r}^{-1}=A_{1};\\
\Psi_{r} \circ B_{1} \circ \Psi_{r}^{-1}=B_{2}; \; \Psi_{r} \circ B_{2} \circ \Psi_{r}^{-1}=B_{3}; \; \Psi_{r} \circ B_{3} \circ \Psi_{r}^{-1}=B_{1};\\
\Psi_{h} \circ A_{1} \circ \Psi_{h}=A_{1}^{-1}; \; \Psi_{h} \circ A_{2} \circ \Psi_{h}=A_{3}^{-1}; \; \Psi_{h} \circ A_{3} \circ \Psi_{h}=A_{2}^{-1};\\
\Psi_{h} \circ B_{1} \circ \Psi_{h}=B_{1}^{-1}; \; \Psi_{h} \circ B_{2} \circ \Psi_{h}=B_{3}^{-1}; \; \Psi_{h} \circ B_{3} \circ \Psi_{h}=B_{2}^{-1}
\rangle=\\
=\langle A_{1},A_{2},A_{3},B_{1},B_{2},B_{3}\rangle \rtimes \langle \Psi_{r},\Psi_{h}\rangle
\cong 
{\mathbb Z}_{3}^{6} \rtimes {\mathfrak S }_{3}.
\end{array}
$$
\qed
\end{example}


\end{document}